\documentclass[a4paper, 12pt]{article}
\pdfoutput=1

\usepackage[T1]{fontenc}
\usepackage[utf8]{inputenc}

\usepackage{amsfonts,amsmath,amssymb,mathtools,dsfont,microtype} % Symbols & microtypography
\usepackage[dvipsnames]{xcolor} % colors with names -- see https://www.overleaf.com/learn/latex/Using_colors_in_LaTeX for list
\usepackage[noBBpl]{mathpazo} % Palatino, but not for mathbb

\DeclareFontFamily{U} {cmr}{}

\DeclareFontShape{U}{cmr}{m}{n}{
	<-6> cmr5
	<6-7> cmr6
	<7-8> cmr7
	<8-9> cmr8
	<9-10> cmr9
	<10-12> cmr10
	<12-> cmr12}{}

\DeclareSymbolFont{Xcmr} {U} {cmr}{m}{n}
\DeclareMathSymbol{\Delta}{\mathord}{Xcmr}{'001}
\DeclareMathSymbol{\Upsilon}{\mathord}{Xcmr}{'007}
\DeclareMathSymbol{\Omega}{\mathord}{Xcmr}{'012}

\usepackage[labelsep = period, labelfont = bf, justification = centering]{caption}
\usepackage{float,graphicx,subcaption}
\DeclareCaptionSubType*[roman]{figure}

\usepackage{booktabs,makecell}
\usepackage{enumitem,mdwlist}
\setlist[itemize]{topsep=0ex,itemsep=0ex,parsep=0.4ex}
\setlist[enumerate]{topsep=0ex,itemsep=0ex,parsep=0.4ex}

\usepackage{pgf,tikz,ifthen,calc}
\usepackage{float, graphicx}
\usepackage{tkz-euclide}
\tkzSetUpPoint[fill=black, size = 3pt]
\tkzSetUpLine[color=black, line width=0.6pt]
\tkzSetUpCircle[color=black, line width=0.6pt]

\usepackage[left = 2.5cm, right = 2.5cm, top = 2.5cm, bottom = 2.5cm, headsep = 12pt, headheight = 15pt]{geometry}
\usepackage{parskip}

\usepackage[hyphens]{url} % Urls together with line breaking them
\usepackage[linktoc = all, hidelinks, colorlinks, unicode=true, pagebackref=true]{hyperref} % Must be loaded after url
\usepackage[capitalise, compress, nameinlink, noabbrev]{cleveref} % Must be loaded after hyperref

\hypersetup{linkcolor={blue!70!black}, citecolor={green!70!black}, urlcolor={blue!70!black}}

\usepackage{amsthm,thmtools,thm-restate}
\declaretheorem[name = Theorem, numberwithin = section, style = plain]{thm}

\declaretheorem[name = Corollary, numberlike = thm, style = plain]{cor}
\declaretheorem[name = Conjecture, numberlike = thm, style = plain]{conj}
\declaretheorem[name = Definition, numberlike = thm, style = definition]{defi}
\declaretheorem[name = Lemma, numberlike = thm, style = plain]{lem}

\declaretheorem[name = Question, numberlike = thm, style = plain]{que}

\crefname{defi}{Definition}{Definitions}
\crefname{thm}{Theorem}{Theorems}
\crefname{lem}{Lemma}{Lemmas}
\crefname{conj}{Conjecture}{Conjectures}
\crefname{claim}{Claim}{Claims}
\crefname{cor}{Corollary}{Corollaries}
\crefname{obs}{Observation}{Observations}
\crefname{prop}{Proposition}{Propositions}
\crefname{que}{Question}{Questions}
\crefname{rem}{Remark}{Remarks}
\crefname{subsection}{\S}{\S\S}

\DeclareFontFamily{U}{matha}{\hyphenchar\font45}
\DeclareFontShape{U}{matha}{m}{n}{
	<5> <6> <7> <8> <9> <10> gen * matha
	<10.95> matha10 <12> <14.4> <17.28> <20.74> <24.88> matha12
}{}
\DeclareSymbolFont{matha}{U}{matha}{m}{n}

\DeclareMathSymbol{\specialuparrow}{\mathrel}{matha}{"D2}
\DeclareMathSymbol{\specialrightarrow}{\mathrel}{matha}{"D1}

\renewcommand*{\backref}[1]{}
\renewcommand*{\backrefalt}[4]{
	\ifcase #1 Not cited.%
	\or $\specialuparrow$#2%
	\else $\specialuparrow$#2%
	\fi%
}

\renewcommand{\epsilon}{\varepsilon}
\renewcommand{\ge}{\geqslant}
\renewcommand{\le}{\leqslant}

\renewcommand{\emptyset}{\varnothing}

\DeclarePairedDelimiter{\ceil}{\lceil}{\rceil}
\DeclarePairedDelimiter{\floor}{\lfloor}{\rfloor}

\newcommand*{\bN}{\mathbb{N}}

\title{Locally bipartite subgraphs\\via multicolor Ramsey numbers}
\author{Raphael Steiner\footnotemark[1]}
\date{\today}

\begin{document}

\maketitle

\renewcommand{\thefootnote}{\fnsymbol{footnote}} % Make affiliation marks symbols

\footnotetext[1]{Department of Mathematics, ETH Z\"{u}rich, Switzerland (\textsf{\href{mailto:raphaelmario.steiner@math.ethz.ch}{raphaelmario.steiner@math.ethz.ch}}). Research supported by the SNSF Ambizione Grant No. 216071.}

\renewcommand{\thefootnote}{\arabic{footnote}} % Return to normal footnote symbols

\begin{abstract}
A famous conjecture of Erd\H{o}s and Hajnal (1969) states that for every integer $g\ge 4$ there exists a (smallest) function $f_g:\mathbb{N}\rightarrow \mathbb{N}$ such that every graph of chromatic number at least $f_g(k)$ contains a subgraph with chromatic number at least~$k$ and girth at least $g$. So far, this has only been proved for $g=4$ by Rödl~(1977), whose proof yields an upper bound on $f_4(k)$ which is a tower of height $\Theta(k^2\log k)$. In this paper, we exhibit a surprising connection of the problem of finding high-chromatic subgraphs of large \emph{odd-girth} (i.e., avoiding short odd cycles) to the problem of lower-bounding multicolor Ramsey numbers of odd cycles. Using this connection, we show a relaxed version of Erd\H{o}s and Hajnal's conjecture: For every odd $g\ge 5$ there is a function $h_g:\mathbb{N}\rightarrow \mathbb{N}$ growing as a power tower of height $\frac{g-3}{2}$ such that every graph of chromatic number at least $h_g(k)$ has a subgraph of chromatic number at least $k$ and odd-girth at least $g$. This proves a conjecture of Mohar and Wu from 2018, addresses a problem of Erd\H{o}s and Hajnal from 1975, and in the $g=5$ case improves the aforementioned tower-type upper bound on $f_4(k)$ due to Rödl to a single-exponential. We then extend this result to a much more general meta-theorem which establishes the analogous statement for many graph parameters $f(\cdot)$. For example, we show that if $f(\cdot)$ is the fractional chromatic number, the Hall ratio, or the strict vector chromatic number (a.k.a. Lov\'{a}sz-Theta-function of the complement), then for every $k,g\in \mathbb{N}$ and every graph $G$ for which $f(G)$ is sufficiently large, there is a subgraph $G'$ of odd-girth at least $g$ such that $f(G')\ge k$.

The key Ramsey-theoretic ingredient of our proof is a new lower bound on Ramsey numbers of odd cycles. More precisely, for $p\ge 1$, let $\mathcal{O}_p=\{C_3,C_5,\ldots,C_{2p+1}\}$. We show that
$R_k(\mathcal{O}_p)\ge (\log^{(p-1)}k)^{k/3-o(k)}$ for every fixed $p$, where $\log^{(p-1)}$ denotes the $(p-1)$-fold iterated logarithm. This yields the first superexponential lower bound on multicolor Ramsey numbers of fixed odd cycles, and extends the recent breakthrough by OpenAI on the multicolor Ramsey numbers of triangles.
\end{abstract}

\section{Introduction}
The chromatic number is one of the fundamental graph parameters, and has been investigated in graph theory at least since the formulation of the Four-Color-Problem by Francis Guthrie in 1852. In particular, the study of which substructures can be guaranteed in all graphs of sufficiently high chromatic number has a long history in graph theory, going back at least to Hadwiger's famous conjecture~\cite{MR12237} from 1943 and the seminal constructions of Tutte and Zykov~\cite{zykov,tutte} from the mid-20th century, demonstrating the existence of triangle-free graphs with arbitrarily large chromatic number. Later, in one of the first instances of the probabilistic method in graph theory, Erd\H{o}s~\cite{MR102081} famously proved the existence of graphs that simultaneously have arbitrarily large chromatic number and girth\footnote{Recall that the \emph{girth} of a graph is defined to be the length of its shortest cycle if it has a cycle, and is set to $\infty$ otherwise.}. Since then, the topic of unavoidable substructures in high-chromatic graphs has flourished into an active and diverse branch of graph theory; we refer the interested reader to~\cite{MR4174126,MR4680419} for two recent surveys on the topic.

Amongst the most famous open problems in this direction is a conjecture due to Erd\H{o}s and Hajnal~\cite{Erdos1969Chromatic} going back to at least 1969, stating that high-chromatic graphs of large girth not only exist (as shown by Erd\H{o}s), but are in fact \emph{ubiquitous} in the sense that they can be found within \emph{every} graph of sufficiently large chromatic number.
\begin{conj}\label{con:erdhajnal}
For every integer $g\ge 4$ there exists a (smallest) function $f_g:\mathbb{N}\rightarrow \mathbb{N}$ such that every graph $G$ of chromatic number at least $f_g(k)$ contains a subgraph with chromatic number at least $k$ and girth at least $g$.  
\end{conj}
Conjecture~\ref{con:erdhajnal}, despite its innocent look, has seen little progress over the years. The only solved special case is when $g=4$, a famous result due to R\"{o}dl~\cite{MR469806} from 1977. R\"{o}dl's proof for the existence of $f_4(k)$, while extremely elegant and memorable, produces an enormous upper bound on $f_4(k)$, namely a power tower of $k$-s of height $\Theta(k^2\log k)$ (cf.~\cite{MR5027075}). The conjecture remains open for $g\ge 5$, but by analyzing the so-called \emph{Burling graphs}, Pettie, Tardos and Walczak~\cite{MR5027075} recently achieved the remarkable result that $f_5(k)$, if it exists, must grow at least as fast as a power tower of height linear in~$k$. The weaker problem, whether one can always pass to a subgraph of large chromatic number and large \emph{odd-girth}\footnote{The odd-girth of a graph is defined as the length of its shortest odd cycle, and is set to $\infty$ if no odd cycle exists.}, has also remained open and was explicitly highlighted as a conjecture by Mohar and Wu in 2018~\cite{MR4356461}.

\begin{conj}[cf.~Conjecture~7 in~\cite{MR4356461}]\label{con:moharwu}
For every odd integer $g\ge 5$ there exists a (smallest) function $h_g:\mathbb{N}\rightarrow \mathbb{N}$ such that every graph of chromatic number at least $h_g(k)$ contains a subgraph of chromatic number at least $k$ and odd-girth at least $g$.  
\end{conj}
A variant of this problem was in fact already posed much earlier by Erd\H{o}s and Hajna~\cite{MR409246} in 1975, who asked whether for every infinite cardinal $\mathfrak{m}$ and every odd integer $g\ge 5$ every infinite graph of chromatic number at least $\mathfrak{m}$ contains a subgraph with chromatic number at least $\mathfrak{m}$ and odd-girth at least $g$.

As our first main result, we resolve Conjecture~\ref{con:moharwu} and obtain a single-exponential upper bound on $f_4(k)=h_5(k)$, improving the aforementioned tower-type bound due to R\"{o}dl. As a consequence, we resolve the problem of Erd\H{o}s and Hajnal for $\mathfrak{m}=\aleph_0$.

More precisely, we prove the following statement. Here, for some integer $i\ge 0$ we denote by $\exp^{(i)}(\cdot)$ the $i$-fold iterated exponential function, with $\exp^{(0)}(x):=x$.
\begin{thm}\label{thm:main}
Let $g\ge5$ be a fixed odd integer and $\varepsilon>0$. For every sufficiently large $k\in \mathbb{N}$, every graph $G$ with $\chi(G)\ge
 \exp^{\left(\frac{g-3}{2}\right)}\!\bigl(k^{3+\varepsilon}\bigr)$ contains a subgraph with chromatic number at least $k$ and odd-girth at least $g$. In particular, every graph with chromatic number at least $e^{k^{3+\varepsilon}}$ contains a triangle-free subgraph of chromatic number at least $k$. 
\end{thm}

\begin{cor}\label{cor:infinite}
Every infinite graph $G$ with infinite chromatic number contains a subgraph with infinite chromatic number and odd-girth at least $g$, for every odd integer $g\ge 5$.
\end{cor}

In fact, we generalize Theorem~\ref{thm:main} to a more general meta-theorem which yields that for a much larger family of graph parameters which we call \emph{$\chi$-amenable}, an analogous statement holds true. We defer the formal and slightly technical definition of $\chi$-amenability to Section~\ref{sec:meta}, and instead give several examples of such parameters.

\begin{thm}\label{thm:meta}
Let $f(\cdot)$ be a $\chi$-amenable graph parameter, let $g\ge5$ be a fixed odd integer and let $\varepsilon>0$. Then for every sufficiently large $k\in \mathbb{N}$, every graph $G$ with $f(G)\ge
 \exp^{\left(\frac{g-3}{2}\right)}\!\bigl(k^{3+\varepsilon}\bigr)$ contains a subgraph $G'$ with $f(G')\ge k$ and odd-girth at least $g$.
\end{thm}
The following result shows that in addition to the chromatic number itself, several other well-studied graph parameters are $\chi$-amenable. For each of the parameters, we reference a book, survey or some selected papers on this parameter, which the unfamiliar reader may consult at their convenience. We will give the formal definitions of these parameters in Section~\ref{sec:amenable}.
\begin{thm}\label{thm:amenable}
Each of the following graph parameters is $\chi$-amenable:
\begin{itemize}
\item the \emph{chromatic number} $\chi(\cdot)$,
    \item the \emph{fractional chromatic number} $\chi_f(\cdot)$~\cite{MR1481157},
    \item the \emph{Hall-ratio} $\rho(\cdot)$~\cite{MR1282556,CropperJacobsonGyarfasLehel2000,MR4208096}, and
    \item the \emph{strict vector-chromatic number} (a.k.a. Lov\'{a}sz-Theta function of the complement)~\cite{MR514926,MR1269161,MR1623197}, i.e., the graph parameter $\overline{\vartheta}(\cdot)$ defined as $\overline{\vartheta}(G)=\vartheta(\overline{G})$ where $\vartheta(\cdot)$ denotes the Lov\'{a}sz-Theta function.\footnote{We remark that this is a natural parameter to consider in this context, since Lov\'{a}sz' sandwich-theorem famously states that $\omega(G)\le \overline{\vartheta}(G)\le \chi(G)$, i.e., $\overline{\vartheta}(G)$ interpolates between the clique and chromatic number of a graph.}
\end{itemize}
\end{thm}

By Theorem~\ref{thm:meta}, for all these parameters, it is true that if they are large on some graph, then we can find a locally bipartite (i.e., large odd-girth) subgraph on which they are still large. To the best of our knowledge, this is a new result in each of the listed cases. The chromatic number case directly corresponds to Theorem~\ref{thm:main} (and so in particular, proving Theorems~\ref{thm:meta} and~\ref{thm:amenable} will automatically also prove Theorem~\ref{thm:main}). The case of the fractional chromatic number reproves the main result of Mohar and Wu~\cite{MR4356461} on the existence of triangle-free subgraphs of high fractional chromatic number, and provides an improved quantitative bound (dropping from tower-type to single-exponential). It also proves the odd-girth variant of their fractional analogue of the Erd\H{o}s-Hajnal conjecture (cf.~Conjecture~4 in~\cite{MR4356461}).

To prove Theorem~\ref{thm:meta}, we create a surprising link between the previously discussed problems and the problem of finding superexponential \emph{lower} bounds on the multicolor Ramsey number of the collection of the first $p$ odd cycles $\mathcal{O}_p:=\{C_3,C_5,\ldots,C_{2p+1}\}$, and along the way prove a novel bound on this Ramsey number. Concretely, extending a recent breakthrough by OpenAI~\cite{openai} for the case $p=1$ (i.e., for the multicolor Ramsey number of the triangle), we prove the following result. 
\begin{thm}\label{thm:ramsey}
Let $p\in \mathbb{N}$ be fixed. Then for all sufficiently large integers $k$ we have
$$R_k(\mathcal{O}_p)\ge 
 (\log^{(p-1)}k)^{k/3-o(k)}.$$
In particular, since $C_{2p+1}\in \mathcal{O}_p$, we have $R_k(C_{2p+1})\ge R_k(\mathcal{O}_p)\ge (\log^{(p-1)}k)^{k/3-o(k)}.$
\end{thm}
Here, for a non-negative integer $i$ we denote by $\log^{(i)}(\cdot)$ the $i$-fold iterated natural logarithm, defined as $\log^{(0)}(x):=x$ and $\log^{(i)}(x):=\log(\log^{(i-1)}(x))$ for $i\ge 1$. 

In the following we give some context for Theorem~\ref{thm:ramsey} and place it within related work. The study of the asymptotic growth of Ramsey numbers is one of the oldest and yet still most active branches of combinatorics and graph theory. We refer the interested reader to~\cite{ramseysurvey,morris2026recentresultsramseytheory} for surveys on the rich subject. The particular case of estimating the growth of multicolor Ramsey numbers $R_k(F)$ of fixed graphs $F$, as the number $k$ of colors grows, has also received a lot of attention. In the following, we discuss specifically the problem of estimating the multicolor Ramsey number $R_k(C_\ell)$ where $\ell\ge 3$ is a fixed odd number and $k$ tends to infinity, which has been the subject of a significant body of prior work.

Until recently, the best asymptotic upper and lower bounds in this setting that hold for arbitrary odd $\ell\ge 3$ were 
$$(2+\delta_\ell)^k\le R_k(C_\ell)\le k^{\frac{2k}{\ell-1}+o(k)},$$ for all sufficiently large $k$ and some constant $\delta_\ell>0$ depending solely on $\ell$. The lower bound is due to Day and Johnson~\cite{DayJohnson2017} and the upper bound as stated was first proved by Axenovich et al.~\cite{AxenovichEtAl2026}, and subsequently improved further by Miyazaki et al.~\cite{miyazaki2026improvedramseyboundsgeneralized} and Huang et al.~\cite{huang2026newupperboundramsey}, however the bounds proved by the latter groups are still of the form $k^{\frac{2k}{\ell-1}+o(k)}$. See also~\cite{BondyErdos1973,ErdosGraham1975,Li2009,LinChen2019,JenssenSkokan2021,Chung1973,Radziszowski2011} for several earlier asymptotic bounds, bounds for different parameter regimes, and related work. 

The case $\ell=3$ has probably received the most attention, and it was a well-known open problem to determine whether or not $R_k(C_3)$ grows exponentially or super-exponentially in $k$. Until recently, the best lower bound for this case was $R_k(C_3)\ge 380^{k/5-O(1)}$ due to Ageron et al.~\cite{ageron2022newlowerboundsschur}. In a recent breakthrough, OpenAI broke this exponential barrier by proving the lower bound
\[
 R_k(C_3)\ge \left(\frac{c k^{1/3}}{\log k}\right)^k
\]
for an absolute constant $c>0$~\cite[Chapter~9, Theorem~1.1]{openai}. Together with the classic factorial upper bound for $R_k(C_3)$~\cite{MR67467} this determines $R_k(C_3)$ as $k^{\Theta(k)}$. Our Theorem~\ref{thm:ramsey} generalizes this result and obtains a superexponential lower bound for arbitrary odd cycles. Our proof builds on and modifies the construction introduced by OpenAI~\cite{openai}.

\paragraph*{Organization.} We start in Section~\ref{sec:meta} by giving a precise definition of the notion of $\chi$-amenability, and then go on to give the proof of our meta-theorem (Theorem~\ref{thm:meta}), conditional on assuming Theorem~\ref{thm:ramsey}. Subsequently, we prove Theorem~\ref{thm:ramsey} in Section~\ref{sec:ramsey}. Afterwards, in Section~\ref{sec:amenable} we prove Theorem~\ref{thm:amenable}, by showing that each of the graph parameters addressed by the theorem is indeed $\chi$-amenable. For the chromatic number $\chi(\cdot)$ itself this will be straightforward (and thus, using Theorem~\ref{thm:meta}, complete the proof of Theorem~\ref{thm:main}). For the other parameters however, $\chi$-amenability is non-trivial, and in each case requires a separate proof. Interestingly, the proof for the case of the fractional chromatic number uses strong duality of linear programming, and the proof for the case of the strict vector chromatic number uses a deep result from the resolution of the Kadison-Singer problem due to Marcus, Spielman and Srivastava~\cite{MR3374963}. Finally, we conclude the paper in Section~\ref{sec:conc} with some concluding remarks and open problems. For completness, we include the quick deduction of Corollary~\ref{cor:infinite} from Theorem~\ref{thm:main} in Appendix~\ref{app}.

\paragraph*{Notation and Terminology.} All graphs in this note are simple, and assumed to be finite by default unless we explicitly label them as infinite. For a graph $G$ we denote by $V(G)$ its vertex-set, by $E(G)$ its edge-set, by $\overline{G}$ its complement and by $\chi(G)$ its chromatic number. Given two graphs $G$ and $H$, we denote by $G\cup H$ the union of $G$ and $H$, i.e. the graph with vertex-set $V(G)\cup V(H)$ and edge-set $E(G)\cup E(H)$. For a finite set $V$, we denote by $K_V$ the complete graph on vertex-set $V$, and refer to all such graphs where $|V|=n$ as $K_n$ if we focus only on isomorphism-invariant properties. For two sets $X$ and $Y$, we denote their symmetric difference by $X\Delta Y$. We also use the symbol $\sqcup$ to denote disjoint set union. All logarithms in this paper are natural, and we use the standard notation $\bN=\{1,2,\ldots\}$ and $[n]=\{1,\ldots,n\}$ for $n\in \bN$, and put $[0]:=\emptyset$.

For a number $d\in \mathbb{N}$, we denote by $\mathbb{R}^{d\times d}$ the space of real $d\times d$-matrices, and by $\mathrm{Sym}_d\subseteq \mathbb{R}^{d\times d}$ the set of all symmetric $d\times d$-matrices. We use the corresponding notation $\mathbb{R}^{S\times S}$ and $\mathrm{Sym}_S$ when we index the rows and columns of our matrices by a finite set $S$. By $I_d$ (or $I_S$, respectively), we denote the identity matrix. A matrix $A\in \mathbb{R}^{d\times d}$ or $A\in \mathbb{R}^{S\times S}$ is \emph{positive semidefinite} if $A$ is symmetric and $\mathbf{x}^\top A\mathbf{x}\ge 0$ for all $\mathbf{x}\in \mathbb{R}^d$ (or $\mathbf{x}\in \mathbb{R}^S$, respectively). We denote this fact by $A\succeq 0$. For any pair of symmetric matrices $A, B\in \mathbb{R}^{d\times d}$ (or $A, B\in \mathbb{R}^{S\times S}$, respectively), we use the notation $A\preceq B$ (or equivalently $B\succeq A$) to indicate the fact that $B-A\succeq 0$. It is well-known and easy to verify that $\preceq$ forms a partial order on $\mathrm{Sym}_d$ (or $\mathrm{Sym}_S$, respectively). Finally, given a matrix $A\in \mathrm{Sym}_d$ (or $A\in \mathrm{Sym}_S$, respectively), we denote by $\lVert A\rVert$ the \emph{spectral norm} and by $\lambda_\mathrm{max}(A)$ the largest eigenvalue of $A$. For every $A\in \mathrm{Sym}_d$ (or $A\in \mathrm{Sym}_S$), it is also well-known that $\lambda_\mathrm{max}(A)$ equals the maximum of the Rayleigh-quotient $\frac{\mathbf{x}^\top A\mathbf{x}}{\mathbf{x}^\top \mathbf{x}}$ over all non-zero (or unit) vectors $x$. Furthermore, we have $\lVert A\rVert=\lambda_\mathrm{max}(A)$ provided $A$ is positive semidefinite.

\paragraph*{AI Disclosure.} We have used ChatGPT 5.6 Pro for proofreading and catching typos. The idea for the proof of Theorem~\ref{thm:ramsey} is due to ChatGPT 5.6 Pro, and a first draft of the proof of that theorem was produced by the same system. However, this initial draft was subsequently heavily restructured and rewritten by the author, and the final proof included in Section~\ref{sec:ramsey} has been carefully checked and edited for correctness and readability by the author. Furthermore, ChatGPT 5.6 Sol Ultra provided a first proof of Lemma~\ref{lem:vartheta}. This initial proof was then significantly simplified by the author and written down in the revised form presented here by the author.

All remaining parts and results of this paper are fully due to the author without any assistance from AI, both in terms of idea-generation and writing. This includes (but is not limited to) identifying and formulating all the statements of the lemmas, theorems and corollaries in this paper, discovering the connection between multicolor Ramsey bounds  of odd cycles and Conjecture~\ref{con:moharwu} as well as the improved bounds for $f_4(k)$, discovering the proofs of Theorems~\ref{thm:main} and~\ref{thm:meta} (as presented in Section~\ref{sec:meta}), coming up with the notion of $\chi$-amenability, and discovering the proofs presented in Section~\ref{sec:amenable}, with the already mentioned exception that an earlier version of the proof of Lemma~\ref{lem:vartheta} is due to ChatGPT 5.6 Sol Ultra.

The high-level vision and direction of the research process underlying the presented results is fully due to the author. Furthermore, the author takes full responsibility for the correctness, presentation and originality of the entire manuscript.

\section{Proof of Theorem~\ref{thm:meta}}\label{sec:meta}
In this section, we give the definition of the notion of a \emph{$\chi$-amenable} graph parameter, and then proceed to prove our meta-theorem, namely Theorem~\ref{thm:meta}, assuming Theorem~\ref{thm:ramsey}.

In the following, a \emph{graph parameter} for us will simply be a function $f$ assigning to every finite graph $G$ a non-negative real number $f(G)$. Roughly speaking, $\chi$-amenable parameters are submultiplicative when taking graph unions, and have the property that if they take on a large value on some graph, then we can find a subgraph where the parameter is still relatively large, while the chromatic number of that subgraph is not much larger than its parameter.

\begin{defi}\label{def:amenability}
Let $f$ be a graph parameter. We say that $f$ is \emph{$\chi$-amenable} if the following two conditions are satisfied.
\begin{enumerate}
    \item For every two graphs $G_1$ and $G_2$ on the same vertex-set, we have $$f(G_1\cup G_2)\le f(G_1)f(G_2).$$
    \item There exists an absolute constant $C\ge 1$ such that every graph $G$ has a subgraph $G'\subseteq G$ with $f(G')\ge f(G)/C$ and $\chi(G')\le Cf(G')$. 
\end{enumerate}
\end{defi}

With the definition of $\chi$-amenability at hand, we are now ready to present the proof of Theorem~\ref{thm:meta}, conditional on the Ramsey-theoretic result stated in Theorem~\ref{thm:ramsey}.

\begin{proof}[Proof of Theorem~\ref{thm:meta}, assuming Theorem~\ref{thm:ramsey}]
In the following, let a $\chi$-amenable graph parameter $f(\cdot)$, an odd number $g\ge 5$ as well as some $\varepsilon>0$ be fixed. Clearly, since the statement we are trying to prove becomes stronger the smaller we choose $\varepsilon$, we may w.l.o.g. assume $\varepsilon\in (0,1)$ in the following. Let $C\ge 1$ be a constant for which $f$ satisfies the second condition in the definition of $\chi$-amenability. Let $k\in \mathbb{N}$ be any given, sufficiently large, integer. Let $p\in \mathbb{N}$ be such that $2p+1=g-2$. Note then that any graph has odd-girth at least $g$ if and only if it is $\mathcal{O}_p$-free.

Now let us consider any graph $G$ such that $f(G)\ge \exp^{\left(\frac{g-3}{2}\right)}(k^{3+\varepsilon})$. Our goal is to prove that $G$ contains a subgraph of odd-girth at least $g$ on which $f$ takes on a value of at least~$k$.

We first note that by our choice of $C$ there exists a subgraph $G'$ of $G$ such that $f(G')\ge  f(G)/C$ and $\chi(G')\le Cf(G')$. In the following, we set $\ell:=\lceil Cf(G')\rceil$. Let us fix a proper coloring $c:V(G')\rightarrow [\ell]$ of $G'$ using at most $\ell$ colors. Let $K\in \mathbb{N}$ be the smallest integer such that $R_K(\mathcal{O}_p)>\ell$  (such a $K$ exists since $R_K(\mathcal{O}_p)\rightarrow \infty$ as $K\rightarrow \infty$, as demonstrated e.g. by Theorem~\ref{thm:ramsey}). Note that since $\ell\rightarrow\infty$ as $k\rightarrow \infty$, with $k$ we may also assume $\ell$, and thus also $K$, to be sufficiently large. By definition of the Ramsey number, $R_K(\mathcal{O}_p)>\ell$ implies that there exists an edge-coloring of the complete graph $K_\ell$ on vertex-set $[\ell]$ with $K$ colors such that there are no monochromatic subgraphs in $\mathcal{O}_p$. Let $H_1,\ldots,H_K$ be the spanning subgraphs of $K_\ell$ such that $H_i$ contains exactly all the edges of color $i$, for $i=1,\ldots,K$. Then each $H_i$ has odd-girth at least $g$. Finally, for each $i\in [K]$ let us define $G_i$ as the spanning subgraph of $G'$ where for every edge $uv\in E(G')$ we have $uv\in E(G_i)$ if and only if $c(u)c(v)\in E(H_i)$. 

We can then observe that for every $i\in [K]$ the graph $G_i$ also has odd-girth at least~$g$: The color-map $c:V(G_i)=V(G')\rightarrow [\ell]=V(H_i)$, by definition of $G_i$, forms a homomorphism from $G_i$ to $H_i$. If $G_i$ contained an odd cycle of length less than $g$, then the image of that cycle under $c$ would yield a closed odd walk of length less than $g$ in $H_i$, within which we could find an odd cycle of length less than $g$ in $H_i$, a contradiction since $H_i$ has odd-girth at least $g$. In the remainder of this proof, our goal is to show that $f(G_i)\ge k$ for some $i$, which will then demonstrate that $G'$, and thus $G$, indeed has a subgraph of odd-girth at least $g$ on which $f$ takes on a value of at least $k$, as desired. 

Towards a contradiction, suppose that $f(G_i)<k$ for every $i$. Then, since $G'$ is the edge-disjoint union of the graphs $G_1,\ldots,G_K$ and since $f$ is $\chi$-amenable, we obtain that:
$$\ell\le 2Cf(G')\le 2C\prod_{i=1}^{K}f(G_i)<2Ck^K.$$ 

By our choice of $K$, we have $R_{K-1}(\mathcal{O}_p)\le \ell$. Using Theorem~\ref{thm:ramsey} and that $K$ is sufficiently large, this implies that
$$(\log^{(p-1)} K)^{K/3-\frac{\varepsilon}{100}\cdot K}\le (\log^{(p-1)} (K-1))^{(K-1)/3-o(K-1)}\le R_{K-1}(\mathcal{O}_p)\le \ell<2Ck^{K}.$$

Taking the $K$-th root, we find that 

$$(\log^{(p-1)} K)^{1/3-\varepsilon/100}<(1+o(1))k$$ and thus $$K<\exp^{(p-1)}\left(k^{1/(1/3-\varepsilon/100)+o(1)}\right)\le \exp^{(p-1)}\left(k^{3+\varepsilon/2}\right).$$ 
Using the well-known upper bound $R_K(\mathcal{O}_p)\le R_K(C_3)\le 3K!$ (cf.~\cite{MR67467}), we now obtain 

$$\exp^{(p)}(k^{3+\varepsilon})=\exp^{\left(\frac{g-3}{2}\right)}(k^{3+\varepsilon})\le f(G)\le Cf(G')\le \ell<R_K(\mathcal{O}_p)\le 3K!\le 3\cdot K^K$$ $$<\exp(K^{1+\varepsilon/100})\le \exp\left(\left(\exp^{(p-1)}(k^{3+\varepsilon/2})\right)^{1+\varepsilon/100}\right)<\exp^{(p)}(k^{3+\varepsilon}),$$ the desired contradiction\footnote{The final inequality holds for all sufficiently large $k$: when $p=1$, it follows from the inequality $(3+\varepsilon/2)(1+\varepsilon/100)<3+\varepsilon$,
while for $p\ge 2$ a fixed power of $\exp^{(p-1)}(k^{3+\varepsilon/2})$ is asymptotically dominated by $\exp^{(p-1)}(k^{3+\varepsilon})$.}. This shows that indeed, $G$ must have a subgraph of odd-girth at least $g$ on which $f$ takes a value of at least $k$, concluding the proof.
\end{proof}

\section{Multicolor Ramsey numbers of odd cycles}\label{sec:ramsey}
In this section, we present the proof of our Ramsey-theoretic result, namely Theorem~\ref{thm:ramsey}. To prove this lower bound, we draw from the ideas of the recent recursive construction of triangle-free edge-colorings of large complete graphs by OpenAI~\cite{openai} which they used for the superexponential lower bound on $R_k(C_3)$, and modify this construction for the purpose of avoiding short monochromatic odd cycles. In the next subsection, we first collect the key auxiliary results from~\cite{openai} that we shall be reusing. In the second subsection we then present our modified construction--first on a high level, explaining the rough structure of our new edge-coloring, how it relates to OpenAI's construction, and where the new idea for extending from the exclusion of triangles to the exclusion of short odd cycles lies. Afterwards, we then formally prove a key technical recursion (Lemma~\ref{lem:recursive-amplification}) which can be easily used to prove the desired lower bound on $R_K(\mathcal{O}_p)$ by induction on $p$.

\subsection{Auxiliary results}
In this preliminary section we recall two auxiliary results that were also used in the proof of the lower-bound for multicolor Ramsey numbers by OpenAI~\cite{openai}. 
The following two statements correspond exactly to Lemma~2.2 and Lemma~2.3 in the OpenAI paper (cf.~\cite{openai}, Chapter~9). The first lemma originates from lower bounds on the hat guessing numbers of complete bipartite graphs (cf.~\cite{alon}), while the second lemma is a simple statement about the existence of a large family of equal-size subsets of a finite ground set with pairwise large differences.

\begin{lem}\label{lem:cover}
Let $H$ be a sufficiently large integer, and set
\[
 m=\ceil{2H\log H},\qquad s=m(m+1)+1.
\]
Then there are fixed maps
\[
 f,g:[H]^s\longrightarrow[H]^s
\]
such that, for every $x,y\in[H]^s$, there is a $d\in[s]$ for which
\[
 x_d=(f(y))_d\qquad\text{or}\qquad y_d=(g(x))_d.
\]
\end{lem}

\begin{lem}\label{lem:palettes}
Let $H$ be a sufficiently large integer, and set
\[
 m=\ceil{2H\log H},\qquad s=m(m+1)+1, \qquad  t=s\ceil{\log H},\qquad M_j=jt\quad(0\le j\le H).
\]
Then, for every $1\le j\le H$, there is a family of \emph{palettes}
\[
 \mathcal{P}_j\subseteq\binom{[jt]}{t}
\]
of size 

\[
 B_j:=|\mathcal{P}_j|\ge
 \frac{\binom{jt}{t}}
 {\displaystyle\sum_{d=0}^{s-1}\binom{t}{d}\binom{(j-1)t}{d}}
 \ge \frac{j^t}{s\bigl(e^2j\ceil{\log H}^{\,2}\bigr)^s},
\]

such that distinct $P,Q\in\mathcal{P}_j$ satisfy
\[
 |P\setminus Q|=|Q\setminus P|\ge s.
\]
\end{lem}

These are the only auxiliary results from~\cite{openai} that we shall need. In the next section, we give a self-contained proof of Theorem~\ref{thm:ramsey} based on these statements. 

\subsection{Proof of Theorem~\ref{thm:ramsey}}
In this section, we prove Theorem~\ref{thm:ramsey}. As mentioned previously, the construction is a modification of the construction for the triangle-case from~\cite{openai}, and the main change comes from refining their recursive colorings by a second coordinate, which allows us to successively increase the odd-girth of the color classes by $2$ with every application of the construction. Despite the large overlap with OpenAI's proof, we chose to give a self-contained proof of Theorem~\ref{thm:ramsey} here (quoting only Lemmas~\ref{lem:cover},~\ref{lem:palettes} and repeating large parts of the arguments from~\cite{openai}). The reason is that various properties of their colorings will be needed that are not explicitly stated outside of proofs in~\cite{openai}. 

Put $\mathcal{O}_0=\emptyset$. For $r\ge0$ and $n\ge1$, let $a_r(n)$ be defined as the least integer $a\in\mathbb{N}$ such that there is an edge-coloring of $K_n$ in $a$ colors with no monochromatic member of $\mathcal{O}_r$ (this can be seen as an integer inverse function of the multicolor Ramsey number of $\mathcal{O}_r$). Thus $a_0(n)=1$, and $a_r(n)$ is nondecreasing in $n$. The central new step to proving Theorem~\ref{thm:ramsey} is captured by the recursion stated in the key lemma~\ref{lem:recursive-amplification} below. It essentially shows that one can construct, starting from an edge-coloring of a small complete graph without a monochromatic odd cycle of length at most $2r+1$, an edge-coloring of a much larger complete graph without a monochromatic odd cycle of length at most $2r+3$ and which does not use many more colors than the original coloring.

\begin{lem}\label{lem:recursive-amplification}
Fix an integer $r\ge0$. Let $H$ be a sufficiently large integer, and define
\[
 m=\ceil{2H\log H},\qquad s=m(m+1)+1, \qquad
t=s\ceil{\log H}\] as well as
$$
 N_H=\left\lceil
 \frac{(H!)^{t-s}}
 {s^H\bigl(e^2\ceil{\log H}^{\,2}\bigr)^{sH}}
 \right\rceil.
$$
Then
\begin{equation}\label{eq:recursion}
 a_{r+1}(N_H)\le Ht\,a_r(H+1).
\end{equation}
\end{lem}

Before starting the formal proof, let us give a high-level and informal description of the construction from OpenAI~\cite{openai} for the triangle case and where the modification lies. 

The construction of OpenAI builds an edge-coloring without monochromatic triangles of larger and larger complete graphs recursively through several stages, which are indexed by some counter $j=0,1,\ldots,H$. The coloring of the complete graph obtained at the end of the process, at stage $H$, is the final output. At each stage, the construction invokes the pair of maps $f,g$ with the coordinate-covering property stated in Lemma~\ref{lem:cover}, as well as large families of well-separated palettes as provided by Lemma~\ref{lem:palettes} to define the next coloring based on the previous one. Our proof follows the same general outline.

In the OpenAI construction as well as in our construction, at each stage $j$ of the process, the vertices of the current complete graph are partitioned into vertex-disjoint blocks of equal size, each of which is associated with its individual palette of colors $P\subseteq[jt]$ from the palette-family given by Lemma~\ref{lem:palettes}. The coloring of the complete graph will then be constructed in such a way that edge-colors inside a block avoid the colors in the associated palette $P$ entirely (we say they are \emph{missing}), while the remaining colors may be used (we say they are \emph{active} on the block). Concretely, inside each block one places a copy of the triangle-free edge-coloring obtained from the preceding stage $j-1$, with colors permuted suitably such that they are all active on the respective block.

Throughout the OpenAI construction (and our construction), one crucially maintains, for each appearing color $c$ in the current edge-coloring, a proper \emph{vertex-coloring} (with few colors) of the spanning subgraph of color $c$-edges. These vertex-colorings are crucial to define the edge-colors of edges \emph{between} different blocks. Namely, the color of an edge between two blocks associated with palettes $P$ and $Q$ is defined in a certain way based on the maps $f$ and $g$ applied to the vertex-colors of the endpoints of that edge. This is done in such a way that different edges $uv$ and $u'v$ between two distinct blocks $A\ni u,u'$ and $B\ni v$ have the following crucial property: If both $uv$ and $u'v$ have the same color $c$ that is active on $A$, then $u$ and $u'$ must have the same \emph{vertex-color} with respect to the proper vertex-coloring of the color $c$-edges. OpenAI's proof uses this directly to rule out monochromatic triangles between two blocks. In our proof here, we instead use the same property to rule out monochromatic odd cycles of length $2r+3$ by reducing them, essentially, to monochromatic odd cycles of length $2r+1$ in an inductively assumed coloring of a smaller complete graph. To make this work, we need to refine OpenAI's recursive edge-coloring as described before with a second coordinate based on an inductively assumed edge-coloring without monochromatic odd cycles of length at most $2r+1$. The latter essentially ensures that monochromatic closed odd walks of length $2r+3$ in our newly constructed edge-colorings cannot repeat vertex-colors, which is used in deriving the final contradiction.

We now move on to the formal proof of Lemma~\ref{lem:recursive-amplification}, from which Theorem~\ref{thm:ramsey} can then be deduced by some standard asymptotic estimates.

\begin{proof}[Proof of Lemma~\ref{lem:recursive-amplification}]
Throughout this proof, we fix the notation
\[
 m=\ceil{2H\log H},\qquad s=m(m+1)+1, \qquad  t=s\ceil{\log H},\qquad M_j=jt\quad(0\le j\le H),
\]
which is consistent with the statement of this lemma as well as with those of Lemmas~\ref{lem:cover} and~\ref{lem:palettes}.

Put $q=a_r(H+1)$, and fix an edge-coloring
\[
 \tau:E(K_{H+1})\longrightarrow[q]
\]
with no monochromatic member of $\mathcal{O}_r$. Here, we identify the vertex-set of the complete graph $K_{H+1}$ with $[H+1]$. 

Next we choose maps $f,g: [H]^s\longrightarrow[H]^s$ as in Lemma~\ref{lem:cover} and, for every stage index $j\in[H]$, fix a palette family
\[
 \mathcal{P}_j\subseteq\binom{[M_j]}{t}
\]
as given by Lemma~\ref{lem:palettes}. We write $B_j=|\mathcal{P}_j|$ and set
\[
 n_0=1,\qquad n_j=\prod_{h=1}^jB_h\quad(1\le j\le H).
\]
For each $0\le j\le H$, let us fix a set $V_j$ of size $n_j$ which will serve as the vertex-set of the edge-colored complete graph we construct at stage $j$ of the process. We also fix, for every $1\le j\le H$, a partition $V_j=\bigsqcup_{P\in \mathcal{P}_j}V_P$, where each of the sets $V_P$ has size exactly $n_{j-1}$ and is identified with the vertex-set of the complete graph $K_{V_{j-1}}$. We call $V_P$ the \emph{block} associated with the palette $P\in \mathcal{P}_j$.

For a palette $P\in\mathcal{P}_j$ and some $c\in[M_j]$, let us call $c$ \emph{active} on $V_P$ if $c\notin P$, and \emph{missing} on $V_P$ if $c\in P$.

In the following, for every $0\le j\le H$, we will construct an edge-coloring
\[
 \kappa_j:E(K_{V_j})\longrightarrow[M_j]\times[q],
\]
and, for every $c\in[M_j]$, a map $L_{j,c}:V_j\to[j+1]$ (this corresponds to one of the ``vertex-colorings'' mentioned in the informal proof overview). These objects will satisfy the following invariants.
\paragraph*{Invariants.}
\begin{enumerate}[label=\textup{(\roman*)}]
 \item\label{item:compatibility} For every $0\le j\le H$ and distinct $u,v\in V_j$ with $\kappa_j(uv)=(c,b)$, we have
 \[
  L_{j,c}(u)\ne L_{j,c}(v)
  \qquad\text{and}\qquad
  b=\tau\bigl(L_{j,c}(u)L_{j,c}(v)\bigr).
 \]

 Hence, the second coordinate of the colorings $\kappa_j$ corresponds to the edge-coloring $\tau$ of the smaller complete graph $K_{H+1}$ lifted to $V_j$ via the vertex-color maps $L_{j,c}$.

 \item\label{item:block-structure} For every $1\le j \le H$ and for every $P\in\mathcal{P}_j$ there is a bijection
 \[
  \phi_P:[M_{j-1}]\longrightarrow[M_j]\setminus P
 \]
 such that, for every edge $uu'$ inside $V_P$, we have
 \[
  \kappa_{j-1}(uu')=(d,b)
  \quad\Longleftrightarrow\quad
  \kappa_j(uu')=(\phi_P(d),b).
 \]
  In other words, the first coordinate of the coloring $\kappa_j$ restricted to the internal edges of any block $V_P$ corresponds, after suitable color-permutation, to the first coordinate of the coloring 
  $\kappa_{j-1}$, while the second coordinate is inherited from $\kappa_{j-1}$ directly. Only the colors outside $P$ are used by the first coordinate of $\kappa_j$ on edges inside $V_P$.
  
 Furthermore, for every $u\in V_P$, we have
 \[
  L_{j,\phi_P(d)}(u)=L_{j-1,d}(u)\quad(d\in[M_{j-1}]),
  \qquad
  L_{j,c}(u)=j+1\quad(c\in P).
 \]
 In other words, for every block $V_P$ and every color $c$, if $c$ is active on $V_P$ then the vertex-coloring $L_{j,c}$ on the block $V_P$ is inherited from the previous stage of the process. If $c$ is missing on $V_P$ then $L_{j,c}$ is constantly equal to $j+1$ on the block $V_P$.

 \item\label{item:cross-structure} For every $1\le j\le H$ and any two distinct palettes $P,Q\in\mathcal{P}_j$ the following holds. If $uv$ is an edge with $u\in V_P$ and $v\in V_Q$, then the first coordinate of $\kappa_j(uv)$ belongs to $P\mathbin{\triangle}Q$. Moreover, if $c\in [M_j]$ is active on $V_P$ and missing on $V_Q$, then, for every fixed $v\in V_Q$, all vertices $u\in V_P$ satisfying
 \[
  \kappa_j(uv)\in\{c\}\times[q]
 \]
 have the same $L_{j,c}$-label.

 \item\label{item:odd-girth} The coloring $\kappa_j$ contains no monochromatic member of $\mathcal{O}_{r+1}$.
\end{enumerate}

In the following, for an edge $uv$ that is colored by $\kappa_j$, we will refer to the element of $[M_j]$ which is the first coordinate of $\kappa_j(uv)$ as the \emph{primary color} of $uv$ and to the element of $[q]$ which is the second coordinate of $\kappa_j(uv)$ as its \emph{secondary color}.

Having set up the notation and stated the desired invariants, we now proceed to describing the construction of the colorings $\kappa_j$ and $L_{j,c}$. Afterwards, we prove by induction on $j$ that they indeed satisfy all the aforementioned invariants (i)--(iv). 

\paragraph*{\textbf{Construction.}} $V_0$ by definition consists of a single vertex, so we define $\kappa_0$ as the empty map (there are no edges to color). Since $M_0=0$ by definition, there are no colors at this stage of the process and so the family of vertex-color maps is empty.

Suppose next that $1\le j\le H$ and that $\kappa_{j-1}:E(K_{V_{j-1}})\rightarrow [M_{j-1}]\times [q]$, and all maps $L_{j-1,d}:V_{j-1}\rightarrow [j]$ for all $d\in [M_{j-1}]$ have already been defined. For every $P\in\mathcal{P}_j$, since $|[M_j]\setminus P|=(j-1)t=M_{j-1}$, there exists a bijection from $[M_{j-1}]$ to $[M_j]\setminus P$. Pick and fix such a bijection and call it $
 \phi_P:[M_{j-1}]\longrightarrow[M_j]\setminus P.
$

We now define the vertex-color maps $L_{j,c}$ for $c\in [M_j]$ as follows. Let any $u\in V_j$ be given and let $V_P$ with $P\in \mathcal{P}_j$ be the unique block containing $u$. We then define
\[
 L_{j,c}(u):=
 \begin{cases}
  L_{j-1,\phi_P^{-1}(c)}(u),&c\notin P,\\
  j+1,&c\in P.
 \end{cases}
\]
Next we define $\kappa_j$. For this, consider any edge $uu'$ of $K_{V_j}$. Suppose first that it is internal to some block, and let $P\in \mathcal{P}_j$ be unique such that $u,u'\in V_P$. 
We now define
\[
 \kappa_j(uu'):=(\phi_P(d),b)
\]
where $d$ and $b$ are unique such that $\kappa_{j-1}(uu')=(d,b)$.

It remains to define $\kappa_j$ on edges between distinct blocks. 

As some preparation, we first fix an arbitrary linear ordering $<$ on $\mathcal{P}_j$, and, for every ordered pair $P<Q$ (arbitrarily) choose and fix distinct colors
\[
 a_1^{PQ},\ldots,a_s^{PQ}\in Q\setminus P,
 \qquad
 b_1^{PQ},\ldots,b_s^{PQ}\in P\setminus Q,
\]
which is possible by choice of the collection $\mathcal{P}_j$ (cf.~Lemma~\ref{lem:palettes}). 

With these choices in place, we can now complete the definition of $\kappa_j$: Let $uu'\in E(K_{V_j})$ be any edge which is \emph{not} internal to some block, and let $P, Q\in \mathcal{P}_j$ be the unique distinct palettes such that $u\in V_P$ and $u'\in V_Q$. We assume $P<Q$ without loss of generality (otherwise the definition proceeds the same way after renaming $u$ and $u'$). Now consider the tuples of vertex-colors
\[
 x(u)=\bigl(L_{j,a_d^{PQ}}(u)\bigr)_{d=1}^s,
 \qquad
 y(u')=\bigl(L_{j,b_d^{PQ}}(u')\bigr)_{d=1}^s.
\]
The colors $a_d^{PQ}$ are active on $V_P$ and the colors $b_d^{PQ}$ are active on $V_Q$, so by our previous definition of the vertex-color maps $L_{j,c}(\cdot)$ we have $x(u),y(u')\in[j]^s\subseteq[H]^s$. If
\[
 x_d(u)=(f(y(u')))_d
\]
for some $d$, let $d_0$ be the least such index and put $c=a_{d_0}^{PQ}$. Otherwise, \cref{lem:cover} gives an index $d$ for which
\[
 y_d(u')=(g(x(u)))_d;
\]
let $d_0$ be the least such index and put $c=b_{d_0}^{PQ}$. Finally, we define
\[
 \kappa_j(uu')=
 \bigl(c,\tau(L_{j,c}(u)L_{j,c}(u'))\bigr).
\]
In either case, $c$ is active on exactly one of $V_P, V_Q$ and missing on the other. Hence, by definition of $L_{j,c}(\cdot)$, one of $L_{j,c}(u),L_{j,c}(u')$ lies in $[j]$ and the other equals $j+1$, so the argument of $\tau$ is indeed always an edge of $K_{H+1}$. This completes the description of the recursive construction of the edge- and vertex-color maps. We now move on to the proof of the invariants.

\paragraph*{Verifying the invariants.}
We prove by induction on the counter $j=0,1,\ldots,H$ that the maps $\kappa_j$ and $L_{j,c}$ defined in the previous paragraph indeed satisfy the invariants (i)--(iv) claimed before. The assertions are immediately checked for $j=0$. 

Now consider any $1\le j\le H$, and suppose we already proved the invariants for all the counters $0,1,\ldots,j-1$. All the statements in invariant (ii) follow directly from the construction of the maps, and so this invariant indeed holds for counter $j$.

Next, we verify (i) for counter $j$. To do so, consider any edge $uv$ of $K_{V_j}$ and let $(c,b)\in [M_j]\times [q]$ be the color assigned to $uv$ by $\kappa_j$. 

Suppose first that $uv$ is internal to some block $V_P$. By invariant~(ii), we have that $\kappa_{j-1}(uv)=(d,b)$ where $d$ is unique such that $\phi_P(d)=c$ and $\phi_P:[M_{j-1}]\rightarrow [M_j]\setminus P$ is the fixed bijection from invariant (ii). By invariant (i) for counter $j-1$ which we assume inductively, we now have
\[
 b=\tau\bigl(L_{j-1,d}(u)L_{j-1,d}(v)\bigr)
 \quad\text{and}\quad
 L_{j-1,d}(u)\ne L_{j-1,d}(v).
\]
Note that $c$ is active on $V_P$. Hence, by our definition of $L_{j,c}$ and since $d=\phi_P^{-1}(c)$, we find that $L_{j,c}(u)=L_{j-1,d}(u)$ and $L_{j,c}(v)=L_{j-1,d}(v)$. Combining this with the equation line above implies that invariant (i) indeed holds in this case. 

Next, suppose that $u$ and $v$ lie in different blocks $V_P, V_Q$. Our definition of $\kappa_j(uv)$ for this case then directly implies that $b=\tau(L_{j,c}(u)L_{j,c}(v))$, as desired, and that $c$ is active in one of $V_P, V_Q$ and missing in the other. The latter, together with the definition of $L_{j,c}(\cdot)$, implies that one of $L_{j,c}(u), L_{j,c}(v)$ must be in $[j]$ while the other equals $j+1$. This implies that $L_{j,c}(u)\neq L_{j,c}(v)$. Altogether, this verifies invariant (i) also in the case that $u,v$ belong to different blocks. 

We now move on to invariant (iii). So consider any two distinct palettes $P,Q \in \mathcal{P}_j$ and consider any edge $uv$ with $u\in V_P$ and $v\in V_Q$. It follows from our definition of $\kappa_j(uv)$, that the primary color of the edge $uv$ under $\kappa_j$ always lies in the union of $s$ colors chosen from $Q\setminus P$ and $s$ colors chosen from $P\setminus Q$. Hence, the primary color of $uv$ always lies in the symmetric difference $P\Delta Q$. This verifies the first part of invariant (iii).

Now consider any $c\in [M_j]$ which is active on $V_P$ but missing on $V_Q$ (so $c\in Q\setminus P$), and fix any vertex $v\in V_Q$. We have to show that all $u\in V_P$ for which the primary color of $uv$ under $\kappa_j$ equals $c$ are assigned the same color by $L_{j,c}(\cdot)$. 

To prove this, consider first the case that $P<Q$. If any edge from $V_P$ to $v$ has the first coordinate of $\kappa_j$ equal to $c$, then we must have $c\in \{a_1^{PQ},\ldots,a_s^{PQ}\}$. In fact, there will be a unique $d_0\in [s]$ such that $c=a_{d_0}^{PQ}$ and
$$L_{j,a_{d_0}^{PQ}}(u)=f((L_{j,b_d^{PQ}}(v))_{d=1}^s)_{d_0}.$$ In particular, $u$ must satisfy $L_{j,c}(u)=f((L_{j,b_d^{PQ}}(v))_{d=1}^s)_{d_0}$. Since the right hand side depends only on $v$ and $d_0$ (and since $d_0$ is uniquely determined by $c$), we can see that the color $L_{j,c}(u)$ is uniquely determined, independent of our choice of $u$. 

Similarly, if $Q<P$ and at least one edge from $V_P$ to $v$ has primary color under $\kappa_j$ equal to $c$, then there is a unique index $d_0\in [s]$ such that $c=b_{d_0}^{QP}$ which satisfies 
$$L_{j,c}(u)=g((L_{j,a_d^{QP}}(v))_{d=1}^s)_{d_0}.$$

Again, this shows that $L_{j,c}(u)$ is uniquely determined by $v$ and $c$, independent of our choice of $u$. 

Thus, in both cases we obtain the desired statement claimed in the invariant~(iii), concluding its proof.

So far we have proved invariants (i)--(iii) at stage $j$. It remains to verify invariant (iv), which we do next. Suppose towards a contradiction that a member of $\mathcal{O}_{r+1}$, i.e., an odd cycle of length at most $2r+3$, was a monochromatic subgraph of the edge-colored complete graph $(K_{V_j},\kappa_j)$. Fix such a cycle $C$ and denote its cyclic vertex-sequence as
\[
 v_1v_2\cdots v_\ell v_1,
\]
where $\ell\in\{3,5,\ldots,2r+3\}$ is the length of $C$. Let $(c,b)\in [M_j]\times [q]$ be the color such that every edge of $C$ is assigned color $(c,b)$ by $\kappa_j$. 

Put
\[
 z_i=L_{j,c}(v_i)\qquad(1\le i\le \ell).
\]
By invariant~(i) for counter $j$, cyclically consecutive colors $z_i$ are distinct. Hence, the cyclic sequence of vertex-colors
\[
 z_1z_2\cdots z_\ell z_1
\]
can be seen as a closed odd walk in the complete graph $K_{[j+1]}\subseteq K_{[H+1]}$ all of whose edges have color $b$ under $\tau$. By our initial assumption on the edge-coloring $\tau$ of $K_{[H+1]}$, it contains no monochromatic member of $\mathcal{O}_r$. In particular, there is no monochromatic closed odd walk of length at most $2r+1$ in $(K_{[H+1]},\tau)$, since every closed odd walk yields an odd cycle of length at most the length of the walk. This implies that $\ell=2r+3$ and that $z_1,\ldots,z_{2r+3}$ must be pairwise distinct, as otherwise $z_1z_2\ldots z_{\ell}$ or a suitable cyclic subsequence of it would form a monochromatic closed odd walk of length at most $2r+1$ in $(K_{[H+1]},\tau)$, a contradiction.

Moving on, we distinguish two cases.

First, suppose that $j+1\notin \{z_1,\ldots,z_{2r+3}\}$. By definition of the vertex-coloring $L_{j,c}(\cdot)$, this implies that every vertex of $C$ lies in a block on which $c$ is active. By invariant~(iii) for counter $j$, no edge whose primary color under $\kappa_j$ equals $c$ can join two distinct such blocks. Hence we find that the cycle $C$ must lie entirely within one block $V_P$ on which $c$ is active. But then invariant~(ii) for counter $j$ shows that the cycle $C$ corresponds to a monochromatic odd cycle of length $2r+3$ under the coloring $\kappa_{j-1}$ of the complete graph $K_{V_{j-1}}$ identified with $K_{V_P}$, contradicting our inductive assumption that invariant (iv) holds for counter $j-1$.

Hence we must have $j+1\in \{z_1,\ldots,z_{2r+3}\}$. Since we have shown that $z_1,\ldots,z_{2r+3}$ are pairwise distinct, $j+1$ equals exactly one of the $z_i$. Without loss of generality (possibly after relabeling), we may assume that $z_1=j+1$. Again by our definition of $L_{j,c}(\cdot)$, we then must have that $v_1$ lies in a block $V_Q$ on which $c$ is missing, while all the other vertices $v_2,\ldots,v_{2r+3}$ must lie in blocks on which $c$ is active. Again by invariant~(iii) for counter $j$, the path $v_2v_3\cdots v_{2r+3}$ must then lie in a single block $V_P\neq V_Q$ on which $c$ is active. Both of the edges $v_1v_2$ and $v_1v_{2r+3}$ are assigned primary color $c$ by $\kappa_j$. Hence, invariant~(iii) for counter~$j$ implies that 
\[
z_2= L_{j,c}(v_2)=L_{j,c}(v_{2r+3})=z_{2r+3},
\]
contrary to the fact that we earlier proved that $z_1,\ldots,z_{2r+3}$ are pairwise distinct. This yields the desired contradiction and proves the invariant (iv) for counter~$j$. All in all, this concludes the inductive proof that all invariants (i)--(iv) are satisfied for $j=0,1,\ldots,H$. 

Now consider the coloring $\kappa_H:E(K_{V_H})\rightarrow [M_H]\times [q]$ of the complete graph $K_{V_H}$ on $n_H$ vertices at stage $H$ of the process. It uses a total of at most $M_Hq$ colors and by invariant~(iv) there is no monochromatic member of $\mathcal{O}_{r+1}$ in any color. By Lemma~\ref{lem:palettes}, we furthermore have
\[
 n_H=\prod_{j=1}^HB_j \ge \prod_{j=1}^{H}\frac{j^t}{s\bigl(e^2j\ceil{\log H}^{\,2}\bigr)^s}
 =
 \frac{(H!)^{t-s}}
 {s^H\bigl(e^2\ceil{\log H}^{\,2}\bigr)^{sH}}
\]
 and hence $n_H\ge N_H$. 

 Altogether this shows that $a_{r+1}(N_H)\le a_{r+1}(n_H)\le M_Hq=Htq=Hta_r(H+1)$, which is the statement of the lemma that we set out to prove. This concludes the proof.
\end{proof}

\begin{proof}[Proof of Theorem~\ref{thm:ramsey}]
In the following, let $H$ be a sufficiently large integer, and let $m$, $s$, $t$ and $N_H$ be defined as in the statement of Lemma~\ref{lem:recursive-amplification} and let $M_H:=Ht$. From the definitions and Stirling's formula one easily derives the following asymptotic estimates:
\begin{align}
 s&=(4+o(1))H^2(\log H)^2,\notag\\
 M_H&=(4+o(1))H^3(\log H)^3,\label{eq:parameter-asymptotics}\\
 \log N_H&=(1+o(1))M_H\log H=(4+o(1))H^3(\log H)^4,\label{eq:size-asymptotics}\\
 \log^{(2)}N_H&=(3+o(1))\log H.\notag
\end{align}
Consequently,
\begin{equation}\label{eq:color-asymptotics}
 M_H=(3+o(1))\frac{\log N_H}{\log^{(2)}N_H},
 \qquad
 \frac{\log N_H}{\log N_{H-1}}\longrightarrow1,
\end{equation}
and, for every fixed $i\ge2$,
\begin{equation}\label{eq:log-transfer}
 \log^{(i+1)}N_H=(1+o(1))\log^{(i)}H.
\end{equation}

We now use these estimates as well as Lemma~\ref{lem:recursive-amplification} to prove by induction on $r\ge 1$ that
\begin{equation}\label{eq:color-complexity}
 a_r(n)\le(3+o(1))\frac{\log n}{\log^{(r+1)}n}
 \qquad(n\to\infty).
\end{equation}
Throughout the following, we will use the following reformulation of the inequality proved in Lemma~\ref{lem:recursive-amplification}: $$a_{r+1}(N_H)\le M_Ha_r(H+1)$$ for every fixed $r$ and $H$ sufficiently large. 

We first prove \eqref{eq:color-complexity} in the base case $r=1$. Noting that $a_0(n)=1$, Lemma~\ref{lem:recursive-amplification} implies that $a_1(N_H)\le M_H$ for $H$ sufficiently large. Now consider any sufficiently large integer $n$, and let $H$ be the smallest integer satisfying $n\le N_H$. Clearly, $H\rightarrow \infty$ as $n\rightarrow \infty$ and $N_{H-1}<n\le N_H$. Equation~\eqref{eq:color-asymptotics} now implies that $\log n=(1+o(1))\log N_H$, which in turn implies $\log\log n=(1+o(1))\log\log N_H$. Consequently, using Equation~\eqref{eq:color-asymptotics} we find that
$$a_1(n)\le a_1(N_H)\le M_H=(3+o(1))\frac{\log N_H}{\log\log N_H}=(3+o(1))\frac{\log n}{\log \log n}.$$ This proves~\eqref{eq:color-complexity} for $r=1$, settling the base of the induction.

For the induction step, suppose that $r\ge 1$ is fixed and we already proved that \eqref{eq:color-complexity} holds for $r$. For every sufficiently large integer $H$, by Lemma~\ref{lem:recursive-amplification} we then find:
\[
 a_{r+1}(N_H)
 \le M_Ha_r(H+1)
 \le(3+o(1))\frac{M_H\log H}{\log^{(r+1)}H}
 =(3+o(1))\frac{\log N_H}{\log^{(r+2)}N_H},
\]
where the final equality follows from \eqref{eq:size-asymptotics} and \eqref{eq:log-transfer}. Now, for any sufficiently large integer $n$, let $H$ be the smallest integer such that $N_H\ge n$. So $N_{H-1}<n\le N_H$. Since $\frac{\log N_H}{\log N_{H-1}}\rightarrow 1$ for $H\rightarrow \infty$ and since $n\le N_H$ implies that $H\rightarrow \infty$ as $n\rightarrow \infty$, we find that $\log n=(1+o(1))\log N_H$, implying in particular that 
$$a_{r+1}(n)\le a_{r+1}(N_H)\le (3+o(1))\frac{\log N_H}{\log^{(r+2)}N_H}=(3+o(1))\frac{\log n}{\log^{(r+2)}n}, $$ as desired. This concludes the inductive proof of~\eqref{eq:color-complexity}. 

Finally, fix the integer $p$ from the theorem and a constant $\delta\in(0,1/3)$, and let
\[
 n=\floor{(\log^{(p-1)}k)^{(1/3-\delta)k}}.
\]
As $k\to\infty$,
\[
 \log n=(1/3-\delta+o(1))k\log^{(p)}k,
 \qquad
 \log^{(p+1)}n=(1+o(1))\log^{(p)}k.
\]
Therefore \eqref{eq:color-complexity} gives
\[
 a_p(n)\le(1-3\delta+o(1))k<k
\]
for all sufficiently large $k$. Hence $K_n$ has a $k$-edge-coloring with no monochromatic member of $\mathcal{O}_p$, so $R_k(\mathcal{O}_p)>n$. Since $\delta>0$ was arbitrary, we have shown that
\[
 R_k(\mathcal{O}_p)\ge(\log^{(p-1)}k)^{k/3-o(k)},
\]
as desired. This concludes the proof of the theorem.
\end{proof}

\section{$\chi$-amenable graph parameters}\label{sec:amenable}
This section is devoted to proving Theorem~\ref{thm:amenable}. We split the proof into two subsections based on the graph parameters treated, where we bundle the chromatic number, the fractional chromatic number and the Hall ratio together, since the proofs for their $\chi$-amenability are related. The proof for $\chi_f$ uses linear programming duality, while the proof for the $\chi$-amenability of the strict vector chromatic number $\overline{\vartheta}$ relies on linear algebra as well as a deep theorem of Marcus, Spielman and Srivastava~\cite{MR3374963} from their resolution of the famous Kadison-Singer problem, and is presented in a separate section.

\subsection{Chromatic number, fractional chromatic number and Hall ratio}

In this section, we prove that each parameter $f\in \{\chi,\chi_f,\rho\}$ is $\chi$-amenable. We start by verifying the submultiplicativity condition $f(G_1\cup G_2)\le f(G_1)f(G_2)$ for these parameters. For the chromatic number, this is a well-known fact and can be verified by considering a product-coloring of optimal colorings of $G_1$ and $G_2$. For the fractional chromatic number and Hall ratio, this inequality is also easy to verify, but we include a proof for completeness. In the following we will use the following definitions which can be easily checked to be consistent with those in the literature (cf.~\cite{MR4452953, MR4208096}): Given a graph $G$, its Hall-ratio is $\rho(G):=\max_{\emptyset\neq H\subseteq G}\frac{|V(H)|}{\alpha(H)}$, where $\alpha$ denotes the independence number and the maximum is taken over all nonempty subgraphs $H$ of $G$. The fractional chromatic number $\chi_f(G)$ of a graph $G$ is defined as the optimal value of the following linear program. Here, $\mathcal{I}(G)$ denotes the collection of independent vertex subsets in $G$.

\begin{align*}
\tag{P}
    \text{min} \sum_{I \in \mathcal{I}(G)}&{x_I} \\
    \text{s.t.} \sum_{I\in \mathcal{I}(G): v \in I}&{x_I}\ge 1~~(\forall v \in V(G)), \\
    & x_I \ge 0~~(\forall I \in \mathcal{I}(G)).
\end{align*}

By linear programming duality, it is also equal to the optimal value of the following dual program. 

\begin{align*}
\tag{D}
    \text{max} &\sum_{v\in V(G)}{y_v} \\
    \text{s.t.}~~~~&\sum_{v \in I}{y_v}\le 1~~(\forall I \in \mathcal{I}(G)), \\
    &~~~~~y_v \ge 0~~(\forall v\in V(G)).
\end{align*}
Based on these definitions, we now observe that $\rho$ and $\chi_f$ are indeed submultiplicative.
\begin{lem}\label{lem:productver1}
For all graphs $G_1$ and $G_2$ on the same vertex-set, we have $\rho(G_1\cup G_2)\le \rho(G_1)\rho(G_2)$ and $\chi_f(G_1\cup G_2)\le \chi_f(G_1)\chi_f(G_2)$.
\end{lem}
\begin{proof}
Let $H$ be any nonempty subgraph of $G_1\cup G_2$. Let $H_1$ and $H_2$ be two spanning subgraphs of $H$ such that $H=H_1\cup H_2$ and $H_i\subseteq G_i$ for $i=1,2$. By definition of $\rho(G_1)$, there exists an independent set $I$ in $H_1$ such that $|I|\ge \frac{|V(H_1)|}{\rho(G_1)}$. Now consider the nonempty subgraph $H_2[I]$ of $G_2$. By definition of $\rho(G_2)$, there exists an independent set $J\subseteq I$ in $H_2$ such that $|J|\ge \frac{|V(H_2[I])|}{\rho(G_2)}=\frac{|I|}{\rho(G_2)}\ge \frac{|V(H_1)|}{\rho(G_1)\rho(G_2)}=\frac{|V(H)|}{\rho(G_1)\rho(G_2)}$. Then $J$ is independent in both $H_1$ and $H_2$, and thus independent in $H=H_1\cup H_2$. It follows that $\alpha(H)\ge \frac{|V(H)|}{\rho(G_1)\rho(G_2)}$, or, equivalently, $\frac{|V(H)|}{\alpha(H)}\le \rho(G_1)\rho(G_2)$. Since $H$ was initially chosen as an arbitrary subgraph of $G_1\cup G_2$, this shows that $\rho(G_1\cup G_2)\le \rho(G_1)\rho(G_2)$, as desired. 

Let us now move on to the fractional chromatic number. For $i=1,2$, let us denote by $(x_I^{(i)})_{I\in \mathcal{I}(G_i)}$ an optimal solution to the linear program~(P) for $G_i$. Note that the independent sets in $G_1\cup G_2$ are exactly the intersections of the members of $\mathcal{I}(G_1)$ and $\mathcal{I}(G_2)$.

Now, for each $S\in \mathcal{I}(G_1\cup G_2)$, let us define
$$x_S':=\sum_{\substack{I\in \mathcal{I}(G_1), J\in \mathcal{I}(G_2):\\
I\cap J=S}}x_I^{(1)}x_J^{(2)}.$$

One then easily checks that $$\sum_{S\in \mathcal{I}(G_1\cup G_2)}x_S'=\sum_{I\in\mathcal{I}(G_1), J\in \mathcal{I}(G_2)}x_I^{(1)}x_J^{(2)}=\sum_{I\in \mathcal{I}(G_1)}x_I^{(1)}\cdot \sum_{J\in \mathcal{I}(G_2)}x_J^{(2)}=\chi_f(G_1)\chi_f(G_2).$$
Similarly, for every $v\in V(G_1)=V(G_2)$ we have
$$\sum_{\substack{S\in \mathcal{I}(G_1\cup G_2):\\v\in S}}x_S'=\sum_{\substack{I\in \mathcal{I}(G_1), J\in \mathcal{I}(G_2):\\ v\in I\cap J}}x_I^{(1)}x_J^{(2)}=\sum_{I\in\mathcal{I}(G_1): v\in I}x_I^{(1)}\cdot \sum_{J\in\mathcal{I}(G_2): v\in J}x_J^{(2)}\ge 1\cdot 1=1.$$ Hence, $(x_S')_{S\in\mathcal{I}(G_1\cup G_2)}$ is feasible for the linear program (P) for $G_1\cup G_2$ with objective value $\chi_f(G_1)\chi_f(G_2)$. Since $\chi_f(G_1\cup G_2)$ is by definition the minimum of that program, it follows that $\chi_f(G_1\cup G_2)\le\chi_f(G_1)\chi_f(G_2)$, as desired. This concludes the proof of the lemma.
\end{proof}

Having established that the first condition of $\chi$-amenability is satisfied by every $f\in \{\chi,\chi_f,\rho\}$, it remains to verify the second condition in the definition (for some suitable choice of the absolute constant $C\ge 1$). 

For $\chi$, the condition is trivially satisfied by taking $C:=1$. For the fractional chromatic number and the Hall ratio, the statement requires a proof, and this is supplied by our next lemma. Here, $e$ denotes Euler's number.

\begin{lem}\label{lem:amenver1}
Let $C:=\frac{2}{1-e^{-1}}$. Every graph $G$ has a subgraph $G'$ such that $$\chi_f(G')\ge \chi_f(G)/C \text{ and }\chi(G')\le C\chi_f(G').$$ Similarly, every graph $G$ has a subgraph $G'$ such that $$\rho(G')\ge \rho(G)/C \text{ and }\chi(G')\le C\rho(G').$$ 
\end{lem}
\begin{proof}
We start by giving the proof for the fractional chromatic number. Let $G$ be any given graph, and set $z:=\chi_f(G)\ge 1$. By our definition of the fractional chromatic number and strong linear programming duality, there exists a pair of optimal primal and dual solutions to the linear programs (P) and (D) for the graph $G$ with the same objective value. Let us denote such a pair of optimal solutions by $(x_I)_{I\in \mathcal{I}(G)}$ and $(y_v)_{v\in V(G)}$, such that $$\sum_{I\in \mathcal{I}(G)}x_I=\sum_{v\in V(G)}y_v=\chi_f(G).$$ Let $p_I:=\frac{x_I}{\chi_f(G)}\ge 0$ for every $I\in \mathcal{I}(G)$. Then we have $\sum_{I\in \mathcal{I}(G)}p_I=1$, and so we can define a probability distribution $\mathcal{D}$ on the independent sets of $G$ by including any independent set $I$ with probability exactly $p_I$. Since $\sum_{I\in \mathcal{I}(G): v\in I}p_I=\frac{\sum_{I\in\mathcal{I}(G): v\in I}x_I}{\chi_f(G)}\ge \frac{1}{\chi_f(G)}=\frac{1}{z}$, we find that a random independent set $I$ drawn according to the distribution $\mathcal{D}$ contains any given vertex with probability at least $\frac{1}{z}$. 

Now let $k:=\lceil z\rceil$ and let us consider drawing $k$ independent samples $I_1,\ldots,I_k$ from the distribution $\mathcal{D}$ of independent sets, and defining a subgraph $G'$ of $G$ as $G':=G[I_1\cup\cdots\cup I_k]$. Evidently, regardless of the outcome of the samples, we have $\chi(G')\le k$. Furthermore, by the above, any given vertex of $G$ is contained in $V(G')$ with probability at least $1-(1-1/z)^k\ge 1-\exp(-k/z)\ge 1-e^{-1}$. Hence, by linearity of expectation, the expectation of $\sum_{v\in V(G')}y_v$ is at least $\sum_{v\in V(G)}y_v\cdot (1-e^{-1})=z(1-e^{-1})$. This shows that there exists some outcome of $G'$ for which $\sum_{v\in V(G')}y_v\ge z(1-e^{-1})$ and $\chi(G')\le k$. In the remainder, we fix such an induced subgraph $G'$ of $G$ with these properties.

Now observe that since every independent set in $G'$ is also an independent set in $G$, we have that $(y_v)_{v\in V(G')}$ is a feasible solution for the dual program (D) for the subgraph $G'$. Hence, it follows that 

$$\chi_f(G')\ge \sum_{v\in V(G')}y_v\ge z(1-e^{-1})\ge \chi_f(G)/C,$$ by our definition of $C$. On the other hand, we have $\chi(G')\le k=\lceil z\rceil\le 2z= \frac{2}{1-e^{-1}}z(1-e^{-1})\le C\chi_f(G')$. This shows that $G'$ is a subgraph of $G$ satisfying both properties claimed in the lemma, concluding the proof of the part about the fractional chromatic number.

We now turn to the second part of the claim about the Hall ratio. So let $G$ be any given graph, and let $z:=\rho(G)\ge 1$. Let $H\subseteq G$ be a nonempty subgraph attaining the maximum in the definition of $\rho(G)$, i.e., such that $z=\frac{|V(H)|}{\alpha(H)}$. Set $k:=\lceil z\rceil$ and let $I_1,\ldots,I_k$ be a sequence of disjoint independent sets in $H$ defined as follows: For each $i=1,\ldots,k$, we pick $I_i$ as a maximum-size independent set in the induced subgraph $H-\bigcup_{1\le j<i}I_j$. Let $G':=H[I_1\cup\cdots\cup I_k]$. It is then clear that $\chi(G')\le k$. Also observe that since $G'$ is an induced subgraph of $H$, we have $\alpha(G')\le \alpha(H)=\frac{|V(H)|}{z}$. Next, we want to give a lower bound on the number of vertices of $G'$. For this, observe that by definition of the Hall ratio and by definition of $I_1,\ldots,I_k$, for every $i=1,\ldots,k$ we have that
$$|I_i|=\alpha\left(H-\bigcup_{1\le j<i}I_j\right)\ge \frac{|V(H-\bigcup_{1\le j<i}I_j)|}{\rho(G)}=\frac{1}{z}\left(|V(H)|-\sum_{j=1}^{i-1}|I_j|\right).$$

Rearranging yields that

$$|V(H)|-\sum_{j=1}^i|I_j|\le \left(|V(H)|-\sum_{j=1}^{i-1}|I_j|\right)-|I_i|\le (1-1/z)\left(|V(H)|-\sum_{j=1}^{i-1}|I_j|\right).$$
This immediately yields that
$$|V(H)|-\sum_{j=1}^k|I_j|\le (1-1/z)^k|V(H)|\le \exp(-k/z)|V(H)|\le e^{-1}|V(H)|,$$ so $|V(G')|=\sum_{j=1}^{k}|I_j|\ge (1-e^{-1})|V(H)|$.

All in all, from this we may conclude that $$\rho(G')\ge \frac{|V(G')|}{\alpha(G')}\ge \frac{(1-e^{-1})|V(H)|}{|V(H)|/z}=z(1-e^{-1})\ge \frac{k}{2}(1-e^{-1})=k/C.$$ In particular, we have $\rho(G')\ge \rho(G)(1-e^{-1})\ge \rho(G)/C$ and $\chi(G')\le k\le C\rho(G')$. This concludes the proof of the second part of the lemma. 
\end{proof}

The previous lemma shows that also $\chi_f$ and $\rho$ satisfy the second condition in the definition of $\chi$-amenability, with $C=\frac{2}{1-e^{-1}}$. All in all, this shows that each of $\chi,\chi_f,\rho$ is $\chi$-amenable. 

\subsection{Strict vector chromatic number}

In this section, we prove that also the strict vector chromatic number $\overline{\vartheta}$ is $\chi$-amenable, thereby completing the proof of Theorem~\ref{thm:amenable}. Before we go into the details of this, let us start by giving a convenient definition of $\overline{\vartheta}$ in terms of positive semidefinite matrices. While many different but equivalent definitions of $\overline{\vartheta}$ and $\vartheta$ are known and used in the literature, we choose to present the following definition since it is particularly useful for our goal of showing $\chi$-amenability. The validity of this definition can be immediately verified by applying~\cite[Proposition 11.9]{MR3967118} to the complement of $G$.

\begin{defi}\label{def:vartheta}
Let $G$ be a graph. Let $\mathcal{M}_G$ denote the set of all positive semidefinite matrices $M\in \mathrm{Sym}_{V(G)}$ such that $M_{u,u}=1$ for all $u\in V(G)$ and $M_{u,v}=0$ for all distinct $u,v\in V(G)$ such that $uv\notin E(G)$. Then
$$\overline{\vartheta}(G)=\max\{\lambda_\mathrm{max}(M)|M\in \mathcal{M}_G\}.$$
\end{defi}

The first part of the definition of $\chi$-amenability, namely the submultiplicativity for graph unions, is well-known and can be explicitly found in the literature on the topic:

\begin{lem}[cf.~Corollary~4.5 in~\cite{MR3537033}]
For all graphs $G_1$ and $G_2$ on the same vertex-set, it holds that $$\overline{\vartheta}(G_1\cup G_2)\le \overline{\vartheta}(G_1)\overline{\vartheta}(G_2).$$
\end{lem}

Hence, it remains to verify that $\overline{\vartheta}$ satisfies the second condition in the definition of $\chi$-amenability. We verify this by proving the following lemma, showing that the condition holds with the absolute constant $C:=8$.

\begin{lem}\label{lem:vartheta}
For every graph $G$ there exists a subgraph $G'\subseteq G$ such that $\overline{\vartheta}(G')\ge \overline{\vartheta}(G)/8$ and $\chi(G')\le 8\overline{\vartheta}(G')$.
\end{lem}

The rest of this section will be devoted to the proof of this lemma (which will then complete the proof of Theorem~\ref{thm:amenable}). In the proof of Lemma~\ref{lem:vartheta}, we combine Definition~\ref{def:vartheta} with the following deep result proved by Marcus, Spielman and Srivastava~\cite{MR3374963} in their resolution of the Kadison-Singer problem. We note that the formulation of the following result is slightly different but very easily deduced from Corollary~1.5 in~\cite{MR3374963}. Namely, we only state the result for real-valued vectors, rather than for complex vectors as in~\cite{MR3374963}.

\begin{thm}[Marcus, Spielman and Srivastava, cf.~Corollary~1.5 in~\cite{MR3374963}]\label{thm:spielman}
Let $r, d, m\in \mathbb{N}$, $\delta\in (0,1]$, and let $u_1,\ldots,u_m\in \mathbb{R}^d$ satisfy
$$\sum_{i=1}^m u_iu_i^\top=I_d \text{   and   } u_i^\top u_i\le \delta \text{ for every }i\in [m].$$

Then there is a partition $[m]=T_1\sqcup\cdots\sqcup T_r$ (empty parts allowed) such that 
$$\left\lVert\sum_{i\in T_j}u_iu_i^\top\right\rVert\le \left(\frac{1}{\sqrt{r}}+\sqrt{\delta}\right)^2$$ for every $j\in [r]$.
\end{thm}

Before we can give the proof of Lemma~\ref{lem:vartheta}, we first deduce the following handy consequence of Theorem~\ref{thm:spielman}. In the following, given a finite set $S$, a matrix $A\in \mathbb{R}^{S\times S}$ and a subset $T\subseteq S$, we denote by $A_T$ the matrix obtained from $A$ by replacing all entries $A_{i,j}$ where at least one of $i,j$ does not lie in $T$ by a $0$. In other words, the $T\times T$-submatrix is inherited from $A$, while all other entries are $0$-s.

\begin{lem}\label{lem:matrixpartition}
Let $S$ be a finite set, and let $M\in \mathrm{Sym}_S$ be a positive semidefinite matrix such that $M_{s,s}=1$ for all $s\in S$. Let $q\in \mathbb{N}$. Then there exists a partition $S=T_1\sqcup \cdots \sqcup T_q$ of $S$ (empty parts allowed) such that 
$$\lVert M_{T_i}\rVert\le \left(\sqrt{\frac{\lVert M\rVert}{q}}+1\right)^2$$ for every $i\in [q]$. 
\end{lem}
In the proof of this auxiliary result, we will use the following four well-known linear algebra facts: \begin{enumerate}\item Every positive semi-definite $n\times n$-matrix is the Gram matrix of a collection of $n$ vectors (this can, for instance, be proved via Cholesky-factorization).
\item Every positive semidefinite $n\times n$-matrix can be written as a sum of matrices of the form $\mathbf{u}\mathbf{u}^\top$ where $\mathbf{u}$ is an $n$-dimensional vector (this, in fact, is an easy consequence of the first fact stated above).
\item For any matrix $B$, the two matrices $B^\top B$ and $BB^\top$ have the same non-zero eigenvalues (the proof is elementary).
\item For any two positive semidefinite matrices $U$ and $V$, we have $\lVert U\lVert\le \lVert U+V\rVert$. This follows immediately from the representation of these spectral norms as maximal Rayleigh-quotients.
\end{enumerate}
\begin{proof}[Proof of Lemma~\ref{lem:matrixpartition}]
Since $M\in \mathbb{R}^{S\times S}$ is positive semidefinite by assumption of the lemma, the first linear algebra fact recorded above implies that $M$ is a Gram matrix of a collection of vectors $(\mathbf{b}_s)_{s\in S}$ in $\mathbb{R}^S$. In other words, we have that $M_{s,s'}=\mathbf{b}_s^\top\mathbf{b}_{s'}$ for all $s,s'\in S$. Denoting by $B\in \mathbb{R}^{S\times S}$  the matrix whose column with index $s$ equals $\mathbf{b}_s$, for every $s\in S$, we then have $M=B^\top B$. Note that the matrix $BB^\top$ is also positive semi-definite, and, by the third aforementioned linear algebra fact, has the same non-zero eigenvalues (which must all be positive). This in particular implies that $\lVert BB^\top\rVert=\lambda_{\mathrm{max}}(BB^\top)=\lambda_{\mathrm{max}}(B^\top B)=\lVert B^\top B\rVert=\lVert M\rVert$. Now, consider the matrix $C:=I_S-\frac{1}{\lVert M\rVert}BB^\top$ (this is well-defined: since all diagonal entries of $M$ are $1$-s, we have $\lVert M\rVert\ge 1$). $C$ is clearly symmetric, and for every $\mathbf{x}\in \mathbb{R}^{S}$ we have $$\mathbf{x}^\top C\mathbf{x}=\mathbf{x}^\top \mathbf{x}-\frac{\mathbf{x}^\top(BB^\top)\mathbf{x}}{\lVert M\rVert}\ge \mathbf{x}^\top \mathbf{x}-\frac{\lambda_{\mathrm{max}}(BB^\top)(\mathbf{x}^\top\mathbf{x})}{\lVert M\rVert}\ge 0,$$ where we used that $\lambda_{\mathrm{max}}(BB^\top)\le \lVert M\rVert$ in the last step. Hence, $C\in \mathrm{Sym}_S$ is positive semidefinite. By the second linear algebra fact mentioned above, there are vectors $\mathbf{u}_1,\ldots,\mathbf{u}_k$ for some $k\in \mathbb{N}$ such that $C=\sum_{i=1}^k \mathbf{u}_i\mathbf{u}_i^\top$. Note that since $\mathbf{u}\mathbf{u}^\top=(\frac{1}{2}\mathbf{u})(\frac{1}{2}\mathbf{u})^\top+(\frac{1}{2}\mathbf{u})(\frac{1}{2}\mathbf{u})^\top+(\frac{1}{2}\mathbf{u})(\frac{1}{2}\mathbf{u})^\top+(\frac{1}{2}\mathbf{u})(\frac{1}{2}\mathbf{u})^\top$ and $(\frac{1}{2}\mathbf{u})^\top(\frac{1}{2}\mathbf{u})=\frac{1}{4}(\mathbf{u}^\top\mathbf{u})$ for every $\mathbf{u}\in \mathbb{R}^S$, possibly by suitably increasing $k$ we may assume that we choose $\mathbf{u}_1,\ldots,\mathbf{u}_k$ in this representation of $C$ such that $\mathbf{u}_i^\top\mathbf{u}_i\le \frac{1}{\lVert M\rVert}$ for every $i\in [k]$. Next, for every $s\in S$, let us set $\mathbf{u}_s:=\frac{1}{\sqrt{\lVert M\rVert}}\mathbf{b}_s$. Since $M$ is the Gram-matrix of the $(\mathbf{b}_s)_{s\in S}$ and by assumption on $M$ in the lemma, we then have $\mathbf{u}_s^\top\mathbf{u}_s=\frac{1}{\lVert M\rVert}\mathbf{b}_s^\top\mathbf{b}_s=\frac{M_{s,s}}{\lVert M\rVert}=\frac{1}{\lVert M\rVert}$ for every $s\in S$. Furthermore, by definition of $C$, we have $$\sum_{s\in S}\mathbf{u}_s\mathbf{u}_s^\top+\sum_{i=1}^k\mathbf{u}_i\mathbf{u}_i^\top=\frac{1}{\lVert M\rVert}\sum_{s\in S}\mathbf{b}_s\mathbf{b}_s^\top+C=\frac{1}{\lVert M\rVert}BB^\top+C=I_S.$$

We are therefore now in the position to apply Theorem~\ref{thm:spielman} to the collection of vectors $\{\mathbf{u}_s| s\in S\}\cup \{\mathbf{u}_i|i\in [k]\}$ with the parameters $\delta:=\frac{1}{\lVert M\rVert}$ and $r:=q$. This yields a partition of the index-set\footnote{W.l.o.g., we assume here that $S$ is disjoint from $[k]$.} $S\cup [k]$ into $q$ disjoint (possibly empty) sets $T_1',\ldots, T_q'$, such that
$$\left\lVert\sum_{s\in S\cap T_j'}\mathbf{u}_s\mathbf{u}_s^\top+\sum_{i\in [k]\cap T_j'}\mathbf{u}_i\mathbf{u}_i^\top\right\rVert\le \left(\frac{1}{\sqrt{q}}+\sqrt{\delta}\right)^2$$ for every $j\in [q]$. Set $T_j:=T_j'\cap S$ for every $j\in [q]$, such that $S=T_1\sqcup \cdots \sqcup T_q$. Note that since the matrices $\sum_{s\in T_j}\mathbf{u}_s\mathbf{u}_s^\top=\sum_{s\in S\cap T_j'}\mathbf{u}_s\mathbf{u}_s^\top$ and $\sum_{i\in [k]\cap T_j'}\mathbf{u}_i\mathbf{u}_i^\top$ are both clearly positive semidefinite for every $j\in [q]$, the fourth linear algebra fact stated before the proof of the lemma now implies that
$$\left\lVert\sum_{s\in T_j}\mathbf{u}_s\mathbf{u}_s^\top\right\rVert\le \left(\frac{1}{\sqrt{q}}+\sqrt{\delta}\right)^2$$ for every $j\in [q]$. Finally, recalling that $\mathbf{u}_s=\frac{1}{\sqrt{\lVert M\rVert}}\mathbf{b}_s$ for every $s\in S$ and that $\delta=\frac{1}{\lVert M\rVert}$, this implies that
$$\left\lVert\sum_{s\in T_j}\mathbf{b}_s\mathbf{b}_s^\top\right\rVert\le \lVert M\rVert\cdot\left(\frac{1}{\sqrt{q}}+\frac{1}{\sqrt{\lVert M\rVert}}\right)^2=\left(\sqrt{\frac{\lVert M\rVert}{q}}+1\right)^2$$ for every $j\in [q]$.

Finally, for every fixed $j\in [q]$, note that we can write $\sum_{s\in T_j}\mathbf{b}_s\mathbf{b}_s^\top$ as $B_jB_j^\top$, where $B_j\in \mathbb{R}^{S\times S}$ is defined as the matrix whose column indexed by $s\in S$ equals $\mathbf{b}_s$ if $s\in T_j$, and has zero-entries otherwise. By the third linear algebra fact stated before the proof, $B_jB_j^\top$ further has the same non-zero (i.e., positive) eigenvalues as $B_j^\top B_j$. The latter matrix is easily verified to equal $M_{T_j}$, by virtue of $M$ being the Gram matrix of the $(\mathbf{b}_s)_{s\in S}$. Hence, it follows that
$$\lVert M_{T_j}\rVert=\lVert B_j^\top B_j\rVert=\lambda_{\mathrm{max}}(B_j^\top B_j)=\lambda_{\mathrm{max}}(B_jB_j^\top)=\lVert B_jB_j^\top\rVert\le \left(\sqrt{\frac{\lVert M\rVert}{q}}+1\right)^2$$ for every $j\in [q]$. This shows that the partition $S=T_1\sqcup \cdots \sqcup T_q$ of $S$ satisfies the property required by the lemma statement, and concludes the proof.
\end{proof}

We are now ready for the proof of Lemma~\ref{lem:vartheta}. As explained before, with this proof we then also conclude the proof of Theorem~\ref{thm:amenable}.

\begin{proof}[Proof of Lemma~\ref{lem:vartheta}]
Let $G$ be any given graph, and let us set $z:=\overline{\vartheta}(G)\ge 1$. By Definition~\ref{def:vartheta}, there exists a matrix $M\in \mathcal{M}_G$ such that $\lambda_\mathrm{max}(M)=z$. Since $M\in \mathrm{Sym}_{V(G)}$ is positive semidefinite, we have that $\lVert M\rVert=\lambda_\mathrm{max}(M)=z$. Let $q:=\lceil z\rceil$. Note that since $M\in \mathcal{M}_G$, we may now apply Lemma~\ref{lem:matrixpartition} to $M$ and $q$. We thus find that there exists a partition $V(G)=V_1\sqcup \cdots \sqcup V_q$ of the vertex-set of $G$ such that $$\lVert M_{V_i}\rVert\le \left(\sqrt{\frac{\lVert M\rVert}{q}}+1\right)^2=\left(\sqrt{\frac{z}{\lceil z\rceil}}+1\right)^2\le 4$$ for every $i\in [q]$. Let now $G'$ be the spanning subgraph of $G$, obtained from $G$ by deleting, for every $i\in [q]$, the edges of $G$ which have both their endpoints in $V_i$. Clearly, each $V_i$ is an independent set in $G'$, and hence we have $\chi(G')\le q$.

In the following we will define a matrix $M'$, of which we will then show that it belongs to $\mathcal{M}_{G'}$ and has a suitably large operator norm, so that we obtain a good lower bound on $\overline{\vartheta}(G')$ using Definition~\ref{def:vartheta}. Concretely, we define
$$M':=I_{V(G)}+\frac{1}{4}\left(M-\sum_{i=1}^{q}M_{V_i}\right).$$
As a linear combination of symmetric matrices, $M'$ is still symmetric. Furthermore, note that since $V_1, \ldots, V_q$ partition $V(G)$, we have that the diagonal entries of $M$ and $\sum_{i=1}^q M_{V_i}$ agree. Thus, every diagonal entry of $M'$ equals $1$. 

Note that with $M$ also each of the matrices $M_{V_i}, i=1,\ldots,q$ is positive semi-definite: For every vector $\mathbf{y}\in \mathbb{R}^{V(G)}$, we have $\mathbf{y}^\top M_{V_i}\mathbf{y}=\mathbf{w}^\top M\mathbf{w}\ge 0$, where $\mathbf{w}$ is obtained from $\mathbf{y}$ by replacing all entries with index outside of $V_i$ by zeros. This in particular implies that $\lambda_{\mathrm{max}}(M_{V_i})=\lVert M_{V_i}\rVert\le 4$ for every $i$.

We next claim that $M'$ is positive semidefinite. So consider any non-zero vector $\mathbf{x}\in \mathbb{R}^{V(G)}$. For each $i\in [q]$, let $\mathbf{x}_i$ denote the vector obtained from $\mathbf{x}$ by replacing all entries of $\mathbf{x}$ with index outside $V_i$ by zeroes. We then have:
\begin{align*}
    \mathbf{x}^\top M' \mathbf{x}&=\mathbf{x}^\top \mathbf{x}+\frac{1}{4}\mathbf{x}^\top M\mathbf{x}-\frac{1}{4}\sum_{i=1}^q\mathbf{x}^\top M_{V_i}\mathbf{x}\\
    &\ge \mathbf{x}^\top\mathbf{x}\left(1-\frac{1}{4}\sum_{i=1}^q\frac{\mathbf{x}^\top M_{V_i}\mathbf{x}}{\mathbf{x}^\top\mathbf{x}}\right)+\frac{1}{4}\mathbf{x}^\top M\mathbf{x}\\
    &=\mathbf{x}^\top\mathbf{x}\left(1-\frac{1}{4}\sum_{i=1}^q\frac{\mathbf{x}_i^\top M_{V_i}\mathbf{x}_i}{\mathbf{x}^\top\mathbf{x}}\right)+\frac{1}{4}\mathbf{x}^\top M\mathbf{x}\\
    &\ge \mathbf{x}^\top\mathbf{x}\left(1-\frac{1}{4}\sum_{i=1}^q\frac{4\mathbf{x}_i^\top\mathbf{x}_i}{\mathbf{x}^\top\mathbf{x}}\right)+\frac{1}{4}\mathbf{x}^\top M\mathbf{x}\\
    &= \mathbf{x}^\top\mathbf{x}\left(1-\frac{\sum_{i=1}^q\mathbf{x}_i^\top\mathbf{x}_i}{\mathbf{x}^\top\mathbf{x}}\right)+\frac{1}{4}\mathbf{x}^\top M\mathbf{x}=\frac{1}{4}\mathbf{x}^\top M\mathbf{x}\ge 0.\tag{7}
\end{align*}
Here we used that $\mathbf{x}_i^\top M_{V_i}\mathbf{x}_i\le \lambda_{\mathrm{max}}(M_{V_i})\mathbf{x}_i^\top\mathbf{x}_i\le 4\mathbf{x}_i^\top\mathbf{x}_i$ for every $i\in [q]$ as well as the fact that $\sum_{i=1}^{q}\mathbf{x}_i^\top\mathbf{x}_i=\mathbf{x}^\top\mathbf{x}$ which follows directly from the definition of the vectors $\mathbf{x}_1,\ldots,\mathbf{x}_q$. Hence, $M'$ is indeed positive semidefinite. Finally, consider any two distinct $u,v\in V(G)$ such that $uv\notin E(G')$. Then, by definition of $G'$, we must have $uv\notin E(G)$ or $u,v\in V_i$ for some $i\in [q]$. In the first case, $M\in \mathcal{M}_G$ implies that $M_{u,v}=0$ and hence $(M_{V_j})_{u,v}=0$ for all $j\in [q]$ as well. Using the definition of $M'$ and since $u\neq v$, this implies that also $(M')_{u,v}=0$ in this case, as desired. In the second case, we have $M_{u,v}=(M_{V_i})_{u,v}$ and $(M_{V_j})_{u,v}=0$ for all $j\neq i$. Hence, also in this case we obtain $(M')_{u,v}=0$ from the definition of $M'$. Thus, we have shown that $M'$ has a zero-entry at position $(u,v)$ whenever $u\neq v$ are nonadjacent vertices of $G'$. Summarizing, we have now verified all conditions defining membership of the set $\mathcal{M}_{G'}$ of matrices, and hence we have shown that $M'\in \mathcal{M}_{G'}$. By Definition~\ref{def:vartheta}, we thus find
$$\overline{\vartheta}(G')\ge \lambda_{\mathrm{max}}(M')=\max_{\mathbf{x}\neq 0} \frac{\mathbf{x}^\top M'\mathbf{x}}{\mathbf{x}^\top\mathbf{x}}.$$ To give the desired lower bound on $\overline{\vartheta}(G')$, consider a unit vector $\mathbf{x}_\ast\in \mathbb{R}^{V(G)}$ such that $\mathbf{x}_\ast^\top M\mathbf{x}_\ast=\lambda_{\max}(M)=z$. By inequality~(7), we then find that 
$$\overline{\vartheta}(G')\ge \mathbf{x}_\ast^\top M' \mathbf{x}_\ast
    \ge \frac{1}{4}\mathbf{x}_\ast^\top M\mathbf{x}_\ast=\frac{z}{4}\ge \frac{q}{8}.$$

Altogether, we have thus found a subgraph $G'$ of $G$ satisfying that $\overline{\vartheta}(G')\ge z/4\ge z/8=\overline{\vartheta}(G)/8$ and $\chi(G')\le q\le 8\overline{\vartheta}(G')$. This concludes the proof of the lemma.
\end{proof}

\section{Concluding remarks}\label{sec:conc}

Several comments about Theorem~\ref{thm:main} are in order. First of all, it shows that guaranteeing a $k$-chromatic subgraph avoiding a fixed \emph{odd cycle} requires only a chromatic number bound which is a tower of \emph{bounded height} (independent of $k$), while Pettie, Tardos and Walczak~\cite{MR5027075} recently showed that to avoid $C_4$, a tower whose height \emph{depends on $k$} (even linearly) is necessary. This suggests that preserving high chromatic number while avoiding short even cycles  is significantly more difficult than avoiding short odd cycles. 

The best \emph{lower bounds} for the functions $h_g(\cdot)$ we know come from considering complete host graphs, i.e., by considering upper bounds on the maximum chromatic number of $n$-vertex  graphs with (odd-)girth $g$. Concretely, Denley~\cite{MR1293399} showed that when $g\ge 5$ is an odd number, then every $n$-vertex graph with odd-girth at least $g$ has independence number at least $n^{(g-3)/(g-1)+o(1)}$. By a simple greedy process that repeatedly removes a largest independent set, this implies that such graphs have chromatic number at most $n^{2/(g-1)+o(1)}$. Thus, we find that $h_g(k)\ge k^{(g-1)/2+o(1)}$ for every odd $g\ge 5$, and we are not aware of a better lower bound. This raises the following question, which remains unanswered here, as all our upper bounds are exponential towers of constant height.

\begin{que}
    Fix an odd number $g\ge 5$. Is the function $h_g$ bounded polynomially?
\end{que}

We suspect the answer may be negative.

Finally, it would be interesting to extend the list in Theorem~\ref{thm:amenable} by identifying further meaningful and well-studied graph parameters that are $\chi$-amenable.

\appendix

\section{Proof of Corollary~\ref{cor:infinite}}\label{app}

In this section, for the sake of completness we supply the proof of Corollary~\ref{cor:infinite} about infinite graphs with infinite chromatic number.

\begin{proof}[Proof of Corollary~\ref{cor:infinite}]
Let $G$ be an infinite graph with infinite chromatic number and let $g\ge 5$ be a fixed odd integer. We claim that there exists a sequence of pairwise vertex-disjoint finite subgraphs $(G_i)_{i\ge 1}$ of $G$ such that for every $i\in \mathbb{N}$, we have $\chi(G_i)\ge i$ and $G_i$ has odd-girth at least $g$. Once we have established the claim, we can observe that the disjoint union $\bigcup_{i=1}^\infty G_i$ forms a subgraph of $G$ with infinite chromatic number and odd-girth at least $g$.

We construct the sequence inductively: Suppose that for some $i\ge 1$ we already constructed vertex-disjoint finite subgraphs $G_1,\ldots,G_{i-1}$ of $G$ such that $\chi(G_j)\ge j$ and $G_j$ has odd-girth at least $g$ for every $1\le j<i$, and let us define $G_i$ with the desired properties. For this purpose, first of all note that the induced subgraph $G[\bigcup_{1\le j<i}V(G_j)]$ is finite, and thus has finite chromatic number. It follows that $G':=G-\bigcup_{1\le j<i}V(G_j)$ has infinite chromatic number. By the de Bruijn--Erd\H{o}s theorem, it follows that there exist finite subgraphs of $G'$ with arbitrarily large chromatic number. Theorem~\ref{thm:main} then implies that any such subgraph of $G'$ with suitably large chromatic number must contain a subgraph with chromatic number at least $i$ and odd-girth at least $g$. We define $G_i$ to be such a subgraph. By definition it is vertex-disjoint from $G_1,\ldots,G_{i-1}$, and has all desired properties. Hence we may conclude the proof.
\end{proof}

\fontsize{11pt}{12pt}
\selectfont
	
\hypersetup{linkcolor={red!70!black}}
\setlength{\parskip}{2pt plus 0.3ex minus 0.3ex}

\bibliographystyle{auxfile.bst}
\bibliography{bib.bib}

\end{document}